\documentclass[12pt, a4paper]{article}

\usepackage[utf8]{inputenc} 
\usepackage[T1]{fontenc}    
\usepackage{subcaption} 

\usepackage[a4paper, top=3cm, bottom=3cm, left=3cm, right=3cm, marginparwidth=1.75cm]{geometry}

\usepackage{listings}   
\usepackage{xcolor}     
\usepackage{amsmath}    
\usepackage{amssymb}    
\usepackage{amsthm}     
\usepackage{graphicx}   
\usepackage{hyperref}   
\hypersetup{colorlinks=true, linkcolor=blue, urlcolor=blue, citecolor=red}

\definecolor{codegreen}{rgb}{0,0.6,0}
\definecolor{codegray}{rgb}{0.5,0.5,0.5}
\definecolor{codepurple}{rgb}{0.58,0,0.82}
\definecolor{backcolour}{rgb}{0.95,0.95,0.95}

\lstdefinestyle{pythonstyle}{
    backgroundcolor=\color{backcolour},   
    commentstyle=\color{codegreen},
    keywordstyle=\color{magenta},
    numberstyle=\tiny\color{codegray},
    stringstyle=\color{codepurple},
    basicstyle=\footnotesize\ttfamily, 
    breakatwhitespace=false,         
    breaklines=true,                 
    captionpos=b,                    
    keepspaces=true,                 
    numbers=left,                    
    numbersep=5pt,                  
    showspaces=false,                
    showstringspaces=false,
    showtabs=false,                  
    tabsize=2
}

\theoremstyle{plain} 
\newtheorem{theorem}{Theorem}[section] 
\newtheorem{lemma}[theorem]{Lemma}       

\newtheorem{conjecture}[theorem]{Conjecture}

\theoremstyle{definition} 
\newtheorem{definition}[theorem]{Definition}

\theoremstyle{remark} 
\newtheorem{remark}[theorem]{Remark}

\DeclareMathOperator{\re}{Re}

\begin{document}

\title{A Complex Analogue of Spencer's Six Standard Deviations Theorem and the Complex Banach--Mazur Distance}

\author{Tomasz Kobos\\
Faculty of Mathematics and Computer Science\\
Jagiellonian University in Cracow\\
\small Lojasiewicza 6, 30-348 Krakow, Poland\\
\small\texttt{tomasz.kobos@uj.edu.pl}\\
\and 
Marin Varivoda\\
Faculty of Science - Department of Mathematics\\ 
University of Zagreb\\ 
\small Bijenička cesta 30, 10000 Zagreb, Croatia\\
\small {\texttt{mvarivod.math@pmf.hr}}
}


\maketitle

\begin{abstract}

We investigate a complex analogue of Spencer's Six Standard Deviations Theorem. Specifically, we propose the following conjecture: for any dimension $n \geq 2$, given vectors $a_1, \ldots, a_n \in \mathbb{C}^n$ satisfying $\|a_i\|_{\infty} \leq 1$ for each $i=1, \ldots, n$, there exists a vector $x \in \mathbb{C}^n$ with all coordinates of modulus one such that $|\langle x, a_i \rangle| \leq \sqrt{n}$ for every $i=1, \ldots, n$. The bound of $\sqrt{n}$ is sharp, as demonstrated by the row vectors of any complex $n \times n$ Hadamard matrix. Furthermore, if the conjecture holds in dimension $n$, it implies that the Banach--Mazur distance between the complex $\ell_1^n$ and $\ell_{\infty}^n$ spaces is equal to $\sqrt{n}$. We prove the conjecture for $n =2, 3$, thereby establishing also that $d_{BM}(\ell_1^n, \ell_{\infty}^n) = \sqrt{n}$ for these dimensions. Additionally, we propose a conjecture about the Banach--Mazur distances between complex $\ell_p^n$ spaces and we verify it for $n=2$. This leads to a complete determination of all possible Banach--Mazur distances between complex $\ell_p^2$ spaces.
\end{abstract}

\section{Introduction}

For an integer $n \geq 1$ and the field $\mathbb{K}=\mathbb{R}$ or $\mathbb{K}=\mathbb{C}$, we denote by $\ell_p^n(\mathbb{K})$ the space $\mathbb{K}^n$ equipped with the norm $\| \cdot \|_p$ defined for $x \in \mathbb{K}^n$ as
$$\|x\|_p = \left ( |x_1|^p + \ldots + |x_n|^p \right )^{\frac{1}{p}}.$$
In particular, $\| \cdot \|_2$ and $\| \cdot \|_{\infty}$ denote the Euclidean and the maximum norms, respectively. The unit ball of $\ell_p^n(\mathbb{K})$ is denoted by $\mathcal{B}_p^n(\mathbb{K})$. When the context will be clear, the field $\mathbb{K}$ will be omitted from the notation. By $\langle \cdot, \cdot \rangle$ we will understand the standard inner product in $\mathbb{K}^n$ that is defined for for $x, y \in \mathbb{K}^n$ as
$$\langle x, y \rangle = x_1 \overline{y_1} + \ldots + x_n\overline{y_n}.$$
In particular, we have $\langle x, x \rangle = \|x\|^2_2$ for $x \in \mathbb{K}^n$. For normed spaces $X, Y$ and a linear operator $A: X \to Y$, we denote by $\|A\|_{X \to Y}$ the usual operator norm. In the case of $X = \ell_p^n(\mathbb{K})$ and $Y=\ell_q^n(\mathbb{K})$ we shall write shortly $\|A\|_{p \to q}$ for the norm of $A$.  By a \emph{Hadamard matrix}, we mean an $n \times n$ matrix with the entries in $\mathbb{K}$ of modulus $1$ and pairwise orthogonal column vectors. For a real (or complex) matrix $A$, by $A^*$ we denote the transpose (or Hermitian conjugate) of $A$.

The famous Spencer's Six Standard Deviations Theorem from $1985$ \cite{spencer} is a cornerstone of combinatorial discrepancy theory, addressing the problem of balancing assignments in set systems and vector spaces. For any collection of $n$ subsets $A_1, \ldots, A_n$ of the set $S=\{1, \ldots, n\}$, the theorem guarantees the existence of a two-coloring of the elements of $S$ such that the discrepancy of each set $A_i$ -- the absolute value of the difference between the elements of each color —- does not exceed $6 \sqrt{n}$. While highly celebrated in combinatorics, the original result of Spencer was of a more geometric flavor, belonging to a broad family of vector balancing problems in $\mathbb{R}^n$. The original result goes as follows.

\begin{theorem}[Spencer's Six Standard Deviations Theorem]
\label{thmspencer}
Let $n \geq 1$ be an integer and let $a_1, \ldots, a_n \in \mathbb{R}^n$ be vectors such that $\|a_i\|_{\infty} \leq 1$ for every $i=1, \ldots, n$. Then, there exists a sign vector $x \in \{-1, 1\}^n$ such that $|\langle x, a_i \rangle| \leq 6 \sqrt{n}$ for every $i=1, \ldots, n$.
\end{theorem}

The optimal constant on the right-hand side of the inequality is unknown, but Spencer's original proof actually gives $5.32$. This was recently improved by Pesenti and Vladu \cite{pesenti} to $3.675$. On the other hand, for every dimension $n$ for which an $n \times n$ Hadamard matrix exists, the constant cannot be less than $1$. Indeed, if $a_1, \ldots, a_n \in \mathbb{R}^n$ are column vectors of a Hadamard matrix, then it follows from the Parseval identity that any sign vector $x \in \{-1, 1\}^n$ satisfies
\begin{equation}
\label{parseval}
|\langle x, a_1 \rangle| ^2 + \ldots + |\langle x, a_n\rangle|^2 = n \|x\|_2^2 = n^2,
\end{equation}
which shows that for at least one index $1 \leq i \leq n$ we must have $|\langle x, a_i \rangle| \geq \sqrt{n}$. Thus, the best possible constant in the estimate, that works for all sufficiently large $n$, cannot be less than $1$.

In 1989, Gluskin \cite{gluskin} independently obtained a similar upper bound, approaching Theorem \ref{thmspencer} from the perspective of convex geometry and functional analysis (see also the papers of Banaszczyk \cite{banaszczyk} and Giannopoulos \cite{giann} for related approaches). The main goal of this paper is to explore a complex analogue of Theorem \ref{thmspencer} and its connection with the geometry of complex Banach spaces, particularly the Banach--Mazur distance. In our exploration, the question of the optimal constant actually plays a central role.

A natural extension of Theorem \ref{thmspencer} to the complex setting involves considering complex vectors $a_1, \ldots, a_n \in \mathbb{C}^n$ satisfying $\|a_i\|_{\infty} \leq 1$ for the maximum norm $\| \cdot \|_{\infty}$ in $\mathbb{C}^n$. It was noted by Hajela in $1992$ \cite{hajela} that the same asymptotic bound of $C \sqrt{n}$ can be obtained for complex vectors, even when only relying on real vectors $x \in \{-1, 1\}^n$. This can be easily derived from Theorem \ref{thmspencer} by embedding $\mathbb{C}^n$ into $\mathbb{R}^{2n}$ in an obvious way and using the inequalities $\|\widetilde{x}\|_{\infty} \leq \|x\|_{\infty} \leq \sqrt{2} \|\widetilde{x}\|_{\infty}$ for a vector $x \in \mathbb{C}^n$ and its embedding $\widetilde{x} \in \mathbb{R}^{2n}$. However, the optimal value of the constant $C$ is even less clear in this case. On the other hand, we believe that in the complex version of Spencer's theorem we shall consider, the question of the best constant might have a clear answer in every dimension. Namely, it is natural to consider not only complex vectors $a_i$ but also allow a vector $x$ to have complex coordinates with modulus equal to $1$. A complex number with the modulus $1$ will be called \emph{unimodular} and, similarly, a vector $x \in \mathbb{C}^n$ with unimodular coordinates will be called a \emph{unimodular vector}. Such vectors are sometimes also called \emph{phase vectors} and the term unimodular should not be confused with other meanings of this word used in different areas of mathematics. Our main conjecture goes as follows.

\begin{conjecture}
\label{conj1}
Let $n \geq 1$ be an integer and let $a_1, \ldots, a_n \in \mathbb{C}^n$ be vectors such that $\|a_i\|_{\infty} \leq 1$ for every $i=1, \ldots, n$. Then, there exists a unimodular vector $x \in \mathbb{C}^n$ such that $|\langle x, a_i \rangle| \leq \sqrt{n}$ for every $i=1, \ldots, n$.
\end{conjecture}

As explained previously, the key aspect of the conjecture is the optimal bound in every dimension $n$, rather than just an asymptotic estimate. Clearly, in every dimension $n$ the right-hand side cannot be less than $\sqrt{n}$, as we can repeat the computation in \eqref{parseval} with respect to vectors $a_1, \ldots, a_n$, being the row vectors of an $n \times n$ complex Hadamard matrix, which, in contrast to the real case, exists in every dimension. Thus, the conjecture is tight, with every complex Hadamard matrix attaining equality in this max-min problem. Notably, in dimension three, there are also other equality cases -- see Section \ref{sec3d} for details. It turns out that for $n=3$, the vectors $a_1, a_2, a_3$ for which equality is attained, need not to be unimodular.

Conjecture \ref{conj1} resembles certain famous open problems in the geometry of complex vectors, such as the conjecture on the existence of $n^2$ equiangular lines in $\mathbb{C}^n$, i.e. Zauner's conjecture about existence of SIC-POVMs (see for example \cite{sic-povm}) or the classification of complex Hadamard matrices in every dimension $n$ (see for example \cite{tadej}), the latter being directly connected to our conjecture (due to the equality case). These problems share a sharp contrast with the real case, where the answers for real scalars are far less clear. This difference seems to stem from the fact that there are only two real scalars of modulus $1$, while in the complex case we have a much richer set -- i.e. the complex unit circle. Another well-known problem in complex geometry with a remarkably similar flavor to Conjecture \ref{conj1} is the complex plank problem, solved by Ball in $2001$ \cite{ball}, with a simplified proof recently published by Ortega \cite{ortega}. The complex plank problem is a version of Tarski's plank problem in the space $\mathbb{C}^n$ endowed with the Euclidean norm. Ball's result states that for any vectors $a_1, \ldots, a_N \in \mathbb{C}^n$ with Euclidean norm $1$ and for any non-negative real numbers $t_1, \ldots, t_N$ with $\sum_{i=1}^{N} t_i^2 = 1$, there exists a vector $x \in \mathbb{C}^n$ such that $\|x\|_2=1$ and $|\langle x, a_i \rangle | \geq t_i $ for every $1 \leq i \leq N$. Thus, while Conjecture \ref{conj1} seeks simultaneous minimization of certain inner products, the complex plank problem addresses their simultaneous maximization. Moreover, both results are clearly contrasted with respect to the norm being used -- in the complex plank problem the Euclidean norm plays the central role, while in the Conjecture \ref{conj1} it is the maximum norm. Beyond a similar outlook, we are not aware of any direct connection between these two problems.

Another motivation for Conjecture \ref{conj1}, a main focus of our investigation, is its connection to the geometry of complex Banach spaces. Recall that for two normed spaces $X, Y$ of dimension $n$ over the same field $\mathbb{K}$, the \emph{Banach--Mazur distance} $d_{BM}(X, Y)$ of $X$ and $Y$ is defined as
$$d_{BM}(X, Y) = \inf \|A\|_{X \to Y} \cdot \|A^{-1}\|_{Y \to X},$$
where the infimum is taken over all linear invertible operators $A: X \to Y$. By compactness, the infimum is attained by some linear operator, and the Banach--Mazur distance is a multiplicative distance on the set of equivalence classes of normed spaces of dimension $n$, where isometric spaces are identified. The Banach--Mazur distance, introduced by Banach and Mazur in $1940$s, plays a significant role in functional analysis and the geometry of Banach spaces by providing a quantitative measure of how "similar" two Banach spaces are in terms of their linear structure. We refer the reader to Chapter $9$ of Tomczak-Jaegermann's monograph \cite{ntj} for its fundamental properties.  Although the concept of the Banach--Mazur distance is natural, precisely determining the distance is often extremely challenging. This difficulty is reflected by the fact that the maximal possible distance between two $n$-dimensional normed spaces over $\mathbb{K}$ is known only for $n=2$ and $\mathbb{K}=\mathbb{R}$, except for the trivial case of $n=1$. In general, the maximum distance is bounded above by $n$ and below by $cn$ for some absolute constant $c$ (in both real and complex cases). For the classical real $\ell_1^n$ and $\ell_{\infty}^n$ spaces, the distance is known only for $n \leq 4$. For $n=2$, the spaces are obviously isometric, but for $n=3$ and $n=4$ determining the distance (which is $\frac{9}{5}$ and $2$, respectively) is non-trivial and was done only recently (see \cite{kobosvarivoda}). In general, the distance $d_{BM}(\ell_1^n(\mathbb{R}), \ell_{\infty}^n(\mathbb{R}))$ is known to be asymptotically of order $\sqrt{n}$, but the exact value in general remains elusive (see \cite{xue} and Remark $3.2$ in \cite{kobosvarivoda}). As we shall see, the problem of determining the Banach--Mazur distance between $\ell_1^n$ and $\ell_{\infty}^n$, but considered now in the complex setting, is notably different and may have a clear answer in every dimension. For general $\ell_p^n$ spaces (where $1 \leq p \leq \infty$), the exact distance $d_{BM}(\ell_p^n, \ell_q^n)$ can be determined easily, when $p$ and $q$ are on the same side of $2$, i.e. $1 \leq p, q \leq 2$ or $2 \leq p, q \leq \infty$. In this case, we have a simple formula $d_{BM}(\ell_p^n, \ell_q^n)=n^{\left |\frac{1}{p}-\frac{1}{q} \right |}$. However, when $p$ and $q$ are on different sides of $2$, as in the case of $\ell_1^n$ and $\ell_{\infty}^n$, determining the exact distance is much more difficult. For such $p$, $q$ the distance is known to be of order  $n^{\alpha}$, where $\alpha = \max \left \{ \left | \frac{1}{p}-\frac{1}{2} \right |,  \left | \frac{1}{q}-\frac{1}{2} \right | \right \}$, with $n^{\alpha}$ being the exact upper bound in the complex case. Besides the previously mentioned examples of the distance $d_{BM}(\ell_1^n(\mathbb{R}), \ell_{\infty}^n(\mathbb{R}))$ for $n=2, 3, 4$, we are not aware of any other instances, where the Banach--Mazur distance $d_{BM}(\ell_p^n, \ell_q^n)$ has been determined, when $p$, $q$ are on different sides of $2$. It seems that in the complex setting, the value of the distance $d_{BM}(\ell_1^n(\mathbb{C}), \ell_{\infty}^n(\mathbb{C}))$ has not been published even for $n=2$, as these spaces are not isometric, unlike in the real case. Conjecture \ref{conj1} leads to the following conjecture concerning this distance.

\begin{conjecture}
\label{conj2}
For every integer $n \geq 1$, we have
\begin{equation}
d_{BM}(\ell_{1}^n(\mathbb{C}), \ell_{\infty}^n(\mathbb{C})) = \sqrt{n}.
\end{equation}
\end{conjecture}

This conjecture is weaker than Conjecture \ref{conj1}, i.e. Conjecture \ref{conj1} implies Conjecture \ref{conj2}. In the complex case, we have the upper bound $d_{BM}(\ell_{1}^n(\mathbb{C}), \ell_{\infty}^n(\mathbb{C})) \leq \sqrt{n}$ in every dimension $n \geq 1$, which follows easily from the existence of a complex $n \times n$ Hadamard matrix (see Lemma \ref{lem:lpdist}). Thus, the main difficulty lies in establishing the lower bound on the Banach--Mazur distance. To see how Conjecture \ref{conj1} implies this lower bound, assume that  $A: \mathbb{C}^n \to \mathbb{C}^n$ is a linear operator such that $\|A\|_{1 \to \infty}=1$. If $e_i$ denotes the $i$-th vector of the canonical unit basis, then $\|A(e_i)\|_{\infty} \leq \|e_i\|_1=1$ for each $i=1, \ldots, n$, which shows that each entry of the matrix of $A$ has modulus at most $1$, since $A(e_1), \ldots, A(e_n)$ are the column vectors of $A$. Hence, if we denote by $a_1, \ldots, a_n$ the row vectors of $A$, then $\|a_i\|_{\infty} \leq 1$ for every $i=1, \ldots, n$, which shows that the vectors $a_1, \ldots, a_n$ satisfy the assumptions of Conjecture \ref{conj1}. Now, if $x \in \mathbb{C}^n$ is a unimodular vector satisfying $|\langle x, a_i \rangle | \leq \sqrt{n}$ for $i=1, \ldots, n$, then for the conjugate vector $\overline{x} \in \mathbb{C}^n$ we have
$$n=\|\overline{x}\|_1=\|A^{-1}(A(\overline{x}))\|_1 = \left \| A^{-1}\left ( \overline{\langle x, a_1 \rangle}, \ldots \overline{\langle x, a_n \rangle} \right ) \right \|_1$$
 $$\leq \|A^{-1}\|_{\infty \to 1} \cdot \left \| \left ( \overline{\langle x, a_1 \rangle}, \ldots \overline{\langle x, a_n \rangle} \right ) \right \|_{\infty} \leq \|A^{-1}\|_{\infty \to 1} \cdot \sqrt{n},$$
which yields the desired inequality $\|A^{-1}\|_{\infty \to 1} \geq \sqrt{n}$.

We also propose a more general conjecture regarding the Banach--Mazur distance of $\ell_p^n$ spaces with $p$, $q$ on different sides of $2$.

\begin{conjecture}
\label{conj3}
For every integer $n \geq 1$ and all $1 \leq q \leq 2 \leq p \leq \infty$ we have
\begin{equation}
\label{lpupbound}
d_{BM}(\ell_{p}^n(\mathbb{C}), \ell_{q}^n(\mathbb{C})) = n^{\alpha},
\end{equation}
where $\alpha = \max \left \{ \frac{1}{2}-\frac{1}{p}, \frac{1}{q}-\frac{1}{2} \right \}$.
\end{conjecture}

As mentioned before, the upper bound is already known and can be handled easily (see Lemma \ref{lem:lpdist}). Confirming this conjecture would lead to determining alll distances between $\ell_p^n(\mathbb{C})$ and $\ell_q^n(\mathbb{C})$, for any possible $p, q$.

Although both Conjecture \ref{conj1} and Conjecture \ref{conj3} imply Conjecture \ref{conj2}, we are not aware of any straightforward way to deduce Conjecture \ref{conj1} or Conjecture \ref{conj3} from Conjecture \ref{conj2}. Moreover, there does not appear to be any immediate implication between Conjecture \ref{conj1} and Conjecture \ref{conj3}. Our contribution to these conjectures is summarized in the following result.
\begin{theorem}
Conjecture \ref{conj1} holds for $n = 2, 3$. In particular, Conjecture \ref{conj2} also holds for $n=2, 3$. Conjecture \ref{conj3} holds for $n=2$.
\end{theorem}

While the two-dimensional case of Conjecture \ref{conj1} is straightforward, the two-dimensional case of Conjecture \ref{conj3} regarding the distance between $\ell_p^2$ and $\ell_q^2$ relies on a certain volume argument that unfortunately fails in higher dimensions. These results are presented in Section \ref{sec2d}, while the previous Section \ref{sechadamard} includes some preliminary observations on the complex Hadamard matrices and the Banach--Mazur distance between $\ell_p^n$ spaces. The main work of the paper is presented in Section \ref{sec3d}, where the three-dimensional case of Conjecture \ref{conj1} is proved. The argument is significantly more involved than for $n=2$, and the main idea is to formulate the problem as a covering problem of the two-dimensional torus by certain shapes called \emph{toric bodies}. To prove the impossibility of covering the torus, a shiftable grid of $9$ points, based on the third roots of unity, is used. It turns out that for any configuration of three toric bodies, one can shift the grid so that one of the $9$ points remains uncovered, thus proving that the torus cannot be covered.

\section{Preliminaries on complex Hadamard matrices}
\label{sechadamard}

In this short section, we demonstrate how to derive the upper bound \eqref{lpupbound} on the Banach--Mazur distance between the complex $\ell_p^n$ and $\ell_q^n$ spaces using a complex Hadamard matrix, when $p$ and $q$ are on different sides of $2$. We omit the case where $p$ and $q$ are on the same side of $2$, as this is well-known and does not highlight any meaningful distinction between real and complex scalars. The crucial difference between the real and complex settings stems from the fact that a complex Hadamard $n \times n$ matrix exists for every $n$, whereas the existence of real Hadamard matrices depends on the dimension. A classical example of an $n \times n$ Hadamard matrix is the $n$-dimensional \emph{Discrete Fourier Transform (DFT) matrix} $F_n$, defined by
$$(F_n)_{j,k} = \exp\left(2\pi i \frac{jk}{n}\right) \quad \text{ for } \quad 1 \leq j,k \leq n$$
(traditionally, a factor of $\frac{1}{\sqrt{n}}$ is included to make the matrix unitary). In complex dimensions $n=2,3,5$, every complex Hadamard matrix is equivalent (up to permutation of rows and columns and multiplication by diagonal unitary matrices) to the DFT matrix. However, already in dimension $n=4$, there exists a one-parameter infinite family of pairwise non-equivalent complex Hadamard matrices. There is a large body of research on complex Hadamard matrices due to their importance in quantum information theory, but a full classification exists only in low dimensions (see \cite{tadej}). To obtain the upper bound \eqref{lpupbound} on $d_{BM}(\ell_p^n, \ell_q^n)$ for $1 \leq q \leq 2 \leq p \leq \infty$, any Hadamard matrix suffices.

The following lemma is not new; the same estimate with essentially the same proof appears, for example, in Proposition 37.6 of \cite{ntj}. Nevertheless, we believe that for the sake of completeness,  the following short argument using the Riesz-Thorin interpolation theorem is worth recalling.

\begin{lemma}
\label{lem:lpdist}
For any integer $n \geq 1$ and all $1 \leq q \leq 2 \leq p \leq \infty$, we have
\[
d_{BM}(\ell_p^n(\mathbb{C}), \ell_q^n(\mathbb{C})) \leq n^{\alpha},
\]
where $\alpha = \max \left\{ \frac{1}{2} - \frac{1}{p}, \frac{1}{q} - \frac{1}{2} \right\}$.
\end{lemma}

\begin{proof}

By a simple duality argument, we may assume without loss of generality that $\frac{1}{2}-\frac{1}{p} \geq \frac{1}{q}-\frac{1}{2}$, i.e. $\alpha = \frac{1}{2} - \frac{1}{p}$. Indeed, we have $d_{BM}(\ell_p^n(\mathbb{C}), \ell_q^n(\mathbb{C})) = d_{BM}(\ell_{p^*}^n(\mathbb{C}), \ell_{q^*}^n(\mathbb{C}))$, where $p^* = \frac{p}{p-1}$ and $q^* = \frac{q}{q-1}$. Thus, if the opposite inequality holds, we may replace $(p,q)$ with the pair $(q^*, p^*)$, which satisfies $1 \leq q^* \leq 2 \leq p^* \leq \infty$.

Let $H$ be any complex $n \times n$ Hadamard matrix. We note that for every $1 \leq i \leq n$ we have $\|H(e_i)\|_{\infty} = 1$, and hence $\|H\|_{1 \to \infty} \leq 1$. Moreover, $\|H\|_{2 \to 2} = \sqrt{n}$, as $\frac{1}{\sqrt{n}} H$ is a unitary matrix. Thus, the Riesz-Thorin interpolation theorem applied with $\theta = \frac{2}{q^*} \in [0,1]$ yields now
$$\|H\|_{q \to q^*} \leq \|H\|^{1-\theta}_{1 \to \infty} \cdot \|H\|^{\theta}_{2 \to 2}=n^{\frac{\theta}{2}}=n^{\frac{1}{q^*}}=n^{1-\frac{1}{q}}.$$
Therefore, since  $p \geq q^*$, it follows that
$$\|H\|_{q \to p} \leq \|H\|_{q \to q^*} \leq n^{1-\frac{1}{q}}.$$
To upper bound the norm $\|H^{-1}\|_{p \to q}$ we will use standard estimates for the $\ell_p$ norms, following from the H\"older inequality. Clearly, as the matrix $\frac{1}{\sqrt{n}}H$ is unitary, we have $H^{-1}=\frac{1}{n}H^*$, so the matrix $\sqrt{n}H^{-1}$ is also unitary. Thus, for any $x \in \mathbb{C}^n$ we have
\begin{equation}
\label{hadamard}
\frac{\|H^{-1}x\|_q}{\|x\|_p} \leq \frac{n^{\frac{1}{q}-\frac{1}{2}}\|H^{-1}x\|_2}{n^{\frac{1}{p}-\frac{1}{2}}\|x\|_2} = n^{\frac{1}{q} - \frac{1}{p} - \frac{1}{2}},
\end{equation}
which implies $\|H^{-1}\|_{p \to q} \leq n^{\frac{1}{q} - \frac{1}{p} - \frac{1}{2}}$. Altogether, we have
$$d_{BM}(\ell_p^n(\mathbb{C}), \ell_q^n(\mathbb{C})) \leq \|H\|_{q \to p} \cdot \|H^{-1}\|_{p \to q} \leq n^{1-\frac{1}{q}} \cdot n^{\frac{1}{q} - \frac{1}{p} - \frac{1}{2}} = n^{\frac{1}{2} - \frac{1}{p}},$$
as desired.
\end{proof}

The proof above works equally well in real dimensions $n$ for which a real $n \times n$ Hadamard matrix exists (conjectured to be $n=1, 2$ and $n=4k$ for $k \in \mathbb{N}$). In such cases, the same upper bound \eqref{lpupbound} holds for $d_{BM}(\ell_p^n(\mathbb{R}), \ell_q^n(\mathbb{R}))$ when $p$ and $q$ are on different sides of $2$. However, this bound is not true for all $n$: by \cite{kobosvarivoda} we have $d_{BM}(\ell_1^3(\mathbb{R}), \ell_{\infty}^3(\mathbb{R})) = \frac{9}{5} > \sqrt{3}$. Moreover, not all real Hadamard matrices yield the same upper bound. For example, there are five equivalence classes of order 16 Hadamard matrices, established by Hall \cite{hall} (see also \cite[pp.~419--421]{wallis} for the description of $5$ classes found by Hall). Recall that two Hadamard matrices are equivalent if one can be obtained from the other by permuting rows or columns and multiplying rows or columns by -1. Using computer aided computation, one can verify that Hall Class I, II, and III matrices (using terminology from \cite{wallis}) yield an upper bound of $4$, whereas Hall Class IV, V matrices lead to the sharper upper bound $d_{BM}(\ell_1^{16}(\mathbb{R}), \ell_{\infty}^{16}(\mathbb{R})) \leq 3.5$.

In the following remark, we note that Conjecture \ref{conj2}, if true, would lead to an interesting corollary about complex Hadamard matrices and unimodular vectors.

\begin{remark}
If Conjecture \ref{conj2} holds, i.e., $d_{BM}(\ell_1^n(\mathbb{C}), \ell_{\infty}^n(\mathbb{C})) = \sqrt{n}$, then for any $n \times n$ complex Hadamard matrix $H$, there exists a unimodular vector $x \in \mathbb{C}^n$ such that $\frac{1}{\sqrt{n}} H x$ is also unimodular. This follows from the fact that, in such case, the estimates in Lemma \ref{lem:lpdist} must be sharp, and in particular, we must have equality in the estimate $\|H^{-1}\|_{\infty \to 1} \leq \sqrt{n}$, which implies equality in \eqref{hadamard} for some $x$. This forces all entries of $x$ to have equal modulus $c > 0$ and all entries of $H^{-1} x$ to have modulus $\sqrt{n} c$. The claim follows since $n H^{-1} = H^*$ is an arbitrary complex Hadamard matrix.
\end{remark}

In other words, if the claim above fails for some complex $n \times n$ Hadamard matrix, then Conjecture \ref{conj2} (and hence also Conjectures \ref{conj1} and \ref{conj3}) fails. For the DFT matrix, such a unimodular vector $x$ can be given explicitly and the claim verified with a direct computation involving the generalized Gauss sums. It is plausible that a direct proof (without referring to the Banach--Mazur distance) of this claim exists, that works for any complex $n \times n$ Hadamard matrix.

\section{The two-dimensional case}
\label{sec2d}

In this section, we prove the two-dimensional cases of our conjectures. We begin with Conjecture \ref{conj1} for $n=2$, the complex analogue of Spencer's theorem. Although the proof is straightforward, it illustrates the geometric idea that will be significantly extended in the three-dimensional case.

\begin{proof}[Proof of Conjecture \ref{conj1} for $n=2$.]
Let $a_1, a_2 \in \mathbb{C}^2$ satisfy $\|a_1\|_{\infty} \leq 1$ and  $\|a_2\|_{\infty} \leq 1$. Suppose, for contradiction, that for every unimodular vector of the form $x = (1, e^{i \varphi})$ (where $\varphi \in [0, 2\pi)$) at least one of the inequalities $|\langle a_1, x \rangle| > \sqrt{2}$ or $|\langle a_2, x \rangle| > \sqrt{2}$ holds. In particular,
$$|\langle a_1, x \rangle |^2 > 2 \geq \|a_1\|^2_2.$$ 
Thus, if we write $a_1=(a, b)$, then
$$|a + b e^{-i \varphi}|^2 > |a|^2 + |b|^2,$$
which rewrites equivalently as $\re(a\overline{b} e^{i \varphi}) > 0$. This inequality is clearly impossible when $a=0$ or $b=0$. Otherwise, it is satisfied on an open interval of length $\pi$ in $\varphi$, which corresponds to an open semicircle on the complex unit circle. Similarly, the inequality $|\langle a_2, x \rangle| > \sqrt{2}$ is also satisfied on (at most) open semicircle on the complex unit circle. However, the unit circle cannot be covered by the union of two open semicircles. This contradiction completes the proof.
\end{proof}

The remainder of this section is devoted to proving Conjecture \ref{conj3} for $n=2$, i.e., determining the exact Banach--Mazur distance between $\ell^2_p(\mathbb{C})$ and $\ell^2_q(\mathbb{C})$ for $1 \leq q \leq 2 \leq p \leq \infty$. To our knowledge, even the case of $q=1$, $p=\infty$ has not appeared in the literature. The key ingredient of our approach is a volume argument. Let us first note, that it is not difficult to determine the maximal possible volume of a linear copy of the unit ball $\mathcal{B}_1^n(\mathbb{C})$ inside $\mathcal{B}_{\infty}^n(\mathbb{C})$ for any $n \geq 1$. Indeed, this maximal volume is the maximum of $|\det A|$ over all linear operators $A: \mathbb{C}^n \to \mathbb{C}^n$ satisfying $\|A\|_{1 \to \infty} \leq 1$.  In particular, for such $A$ we have $\|A(e_i)\|_2 \leq \sqrt{n} \|A(e_i)\|_{\infty} \leq 1$ for every $i=1, \ldots, n$. Thus, it follows from Hadamard's inequality that $|\det A| \leq n^{\frac{n}{2}}$, with equality attained for any complex Hadamard matrix.

It is much less clear how to determine the maximal volume copy of $\mathcal{B}_{\infty}^n(\mathbb{C})$ inside $\mathcal{B}_1^n(\mathbb{C})$. The real case would suggest that the standard position of these balls should be far from maximizing the volume. It turns out that in the case of $n=2$, the standard position $\frac{1}{2}\mathcal{B}_{\infty}^2(\mathbb{C}) \subseteq \mathcal{C}_1^2(\mathbb{B}) $ is actually optimal, i.e. the maximal possible volume is equal to $\frac{1}{4}$.

\begin{lemma} \label{lem:2d_bm_volume_ineq}
For any complex $2 \times 2$ matrix $A$,
\begin{equation} \label{ineq:2d_bm_vol}
 \|A\|_{\infty \to 1} \geq 2 \sqrt{|\det {A}|}.
\end{equation}
\end{lemma}
\begin{proof}
    Let us write $A = \begin{bmatrix} a & b \\ c & d
    \end{bmatrix}$ for some $a, b, c, d \in \mathbb{C}$. We consider the inequality
$$|a + b e^{i\varphi}|^2 \geq |a|^2 + |b|^2$$
for a variable $\varphi \in [0, 2 \pi]$. Similarly like in the previous proof, we deduce that, in this case, there is a closed interval of length at least $\pi$ in $\varphi$ where this inequality is satisfied (exactly $\pi$ if $a, b \neq 0$). Analogously, there is a closed interval of length at least $\pi$ in $\psi \in [0, 2 \pi]$ such that $|c + d e^{i\psi}|^2 \geq |c|^2+|d|^2$. Since any two closed semicircles on the circle intersect, there exists a common angle $\alpha \in [0, 2\pi]$ such that
$$|a + b e^{i \alpha}| \geq \sqrt{|a|^2 + |b|^2} \quad \text{ and } \quad |c + d e^{i \alpha}| \geq \sqrt{|c|^2 + |d|^2}.$$
Therefore, we can now estimate
\begin{align*}
\|A\|_{\infty \to 1}
&\geq |a + b e^{i \alpha}| + |c + d e^{i \alpha}| \\
&\geq \sqrt{|a|^2 + |b|^2} + \sqrt{|c|^2 + |d|^2} \\
&\geq 2 \sqrt[4]{(|a|^2 + |b|^2)(|c|^2 + |d|^2)} \\
&\geq 2 \sqrt{|\det A|},
\end{align*}
where the second estimate follows from the AM-GM inequality, and the last uses Hadamard's inequality for $ 2\times 2$ matrices.

\end{proof}

With the previous lemma established, we are ready to prove Conjecture \ref{conj3} for $n=2$.

\begin{proof}[Proof of Conjecture \ref{conj3} for $n=2$.]

By the same duality argument as in the beginning of the proof Lemma \ref{lem:lpdist}, we may assume without loss of generality that $\frac{1}{q}-\frac{1}{2} \geq \frac{1}{2}-\frac{1}{p}$. Let us denote $d = d_{BM}(\ell_p^2(\mathbb{C}), \ell_q^2(\mathbb{C}))$. By Lemma \ref{lem:lpdist} we have $d \leq 2^{\frac{1}{q}-\frac{1}{2}}$, so it is enough to prove the reverse inequality. Let $A: \mathbb{C}^2 \to \mathbb{C}^2$ be an operator such that $\|A\|_{q \to p} = 1$, $\|A^{-1}\|_{p \to q} = d$ and let $a_1, a_2 \in \mathbb{C}^2$ be the columns of $A$. By H\"older's inequality
$$1 \geq \|A(e_i)\|_p = \|a_i\|_p \geq 2^{\frac{1}{p}-\frac{1}{2}} \|a_i\|_2,$$
which gives us $\|a_i\|_2 \leq 2^{\frac{1}{2}-\frac{1}{p}}$ for $i=1, 2$. Thus, Hadamard's inequality now yields
$$\sqrt{|\det A|} \leq \sqrt{ \|a_1\|_2 \|a_2\|_2} \leq 2^{\frac{1}{2}-\frac{1}{p}}$$
and hence $\sqrt{|\det A^{-1}|} \geq 2^{\frac{1}{p}-\frac{1}{2}}$. Therefore, Lemma \ref{lem:2d_bm_volume_ineq} applied to $A^{-1}$ gives us $\|A^{-1}\|_{\infty \to 1} \geq 2^{\frac{1}{2}+\frac{1}{p}}$. On the other hand, from the assumptions about the operator $A$ and the standard containments between the unit balls of the $\ell_p$ spaces, it follows that
$$2^{-\frac{1}{p}} A^{-1}(\mathcal{B}_{\infty}^2) \subseteq A^{-1}(\mathcal{B}_p^2) \subseteq d \mathcal{B}_q^2 \subseteq d2^{1-\frac{1}{q}} \mathcal{B}_1^2,$$
and therefore
$$A^{-1} (\mathcal{B}_{\infty}^2) \subseteq d 2^{1+\frac{1}{p}-\frac{1}{q}} \mathcal{B}_1^2.$$
Hence
$$d \cdot 2^{1+\frac{1}{p}-\frac{1}{q}} \geq \|A^{-1}\|_{\infty \to 1} \geq 2^{\frac{1}{2}+\frac{1}{p}},$$
which simplifies to $d \geq 2^{\frac{1}{q} - \frac{1}{2}}$, as required.

\end{proof}

If the standard position $\frac{1}{n} \mathcal{B}_{\infty}^n(\mathbb{C}) \subseteq \mathcal{B}_1^n(\mathbb{C})$ would be volume-optimal for all $ n \geq 2$, then Conjecture \ref{conj3} would follow with the same argument as for $n=2$. However, this fails in any dimension $n \geq 3$, as shown in the following remark.

\begin{remark}
For every $n \geq 3$ there exists a complex $n \times n$ matrix $A$ such that
$$ \|A\|_{\infty \to 1} < n \sqrt[n]{|\det{A}|}.$$

We construct a counterexample inductively. For the base case of $n=3$, we used numerical optimization to find a complex $3 \times 3$ matrix with $|\det A_3| = 1$ such that $\|A_3\|_{\infty \to 1} < 3$. The resulting matrix is (rounded to 4 decimal places):
$$
A_3 \approx 
\begin{bmatrix} 
-0.3857 - 0.5131i &  0.1795 - 0.5142i &  0.0675 - 0.5393i \\
-0.0179 - 0.6519i & -0.6798 + 0.2941i &  0.0611 + 0.1143i \\
-0.3182 - 0.2582i &  0.2765 - 0.2900i & -0.1519 + 0.8092i 
\end{bmatrix}.
$$
Numerical computation yields $\|A_3\|_{\infty \to 1} \approx 2.9978 < 3$.

Suppose now, that for $n \geq 3$ we have constructed a complex $n \times n$ matrix $A_n$ such that $| \det A_n | = 1$ and $\|A_n\|_{\infty \to 1} < n$. Let $A_{n+1}$ be $(n+1) \times (n+1)$ matrix defined as
$$A_{n+1}:=\left[ 
\begin{array}{c|c} 
  1 & 0 \\ 
  \hline 
  0 & A_n 
\end{array} \right].$$
For any vector $x = \left( x_0, x_1, \ldots, x_n\right) \in \mathbb{C}^{n+1}$ with $\|x\|_\infty \leq 1$, we now have 
\[
\|A_{n+1}x\|_1 = |x_0| +\|A_n \left(x_1, \ldots, x_n) \right)\|_1 < 1 + n
\]
where we used $|x_0| \leq 1$ and $\|A_n\|_{\infty \to 1} < n$. Therefore $\|A_{n+1}\|_{\infty \to 1} < n+1$ as desired.
\end{remark}

\section{The three-dimensional case}
\label{sec3d}

In this section we prove Conjecture \ref{conj1} for $n=3$. The main idea of the proof is as follows. In the two-dimensional case (Section \ref{sec2d}), we showed that if the conjecture fails, then the unit circle can be covered by the union of two open semicircles, which is obviously impossible. Similarly, for $n=3$ we will establish that the assumption that the conjecture does not hold implies that the two-dimensional torus, i.e., the product of two circles, can be covered by the union of three toric sets, which we will call toric bodies. Hence, our main goal will then be to prove that such a covering is impossible.

Let the torus be given as  $\mathbb{T} = \mathbb{R}^2 / (2\pi \mathbb{Z})^2$, so that points $(\theta_1, \theta_2) \in \mathbb{R}^2$ and $ (\theta'_1, \theta'_2) \in \mathbb{R}^2$ are identified if $\theta'_1 = \theta_1 + 2k\pi$ and $\theta'_2 = \theta_2 + 2l\pi$ for some $k, l \in \mathbb{Z}$.

\begin{definition}
    Let $a \in \mathbb{C}^3$ be a vector satisfying $\|a\|_\infty \leq 1$. We define the \emph{toric body} centered in $a$ as the set
$$B(a) = \left \{ (\theta_1, \theta_2) \in \mathbb{T} : \left | \left \langle a,  (1, e^{i \theta_1}, e^{i \theta_2}) \right \rangle \right | > \sqrt{3} \right \}.$$
A similar set defined with weak inequality ($\geq$) instead is called a \emph{closed toric body}. If the vector $a$ is unimodular, then $B(a)$ is called a \emph{unimodular toric body}.
\end{definition}

With this terminology, the $3$--dimensional case of Conjecture~\ref{conj1} can be restated as follows: for any vectors $a_1, a_2, a_3 \in \mathbb{C}^3$ with $\|a_i\|_{\infty} \leq 1$, the union of the toric bodies centered in these vectors fails to cover the torus. In other words, we have
$$\mathbb{T} \not \subseteq B(a_1) \cup B(a_2) \cup B(a_3).$$

Note that closed toric bodies can cover the torus -- for example, when $a_1, a_2, a_3$ are the row vectors of a $3\times 3$ complex Hadamard matrix, equality holds in \eqref{parseval}, so the closed bodies centered in $a_i$'s cover $\mathbb{T}$.

To prove impossibility of a covering of the torus with open toric bodies, we introduce a grid $\mathcal{G} \subseteq \mathbb{T}$ of $9$ points defined as
    \begin{equation}
        \mathcal{G} = \left \{ \left [ \left(\frac{2\pi}{3} \cdot j, \frac{2\pi}{3} \cdot k\right) \right ] \; : \; j, k = -1, 0, 1 \right \}.
    \end{equation}

Since the torus is two-dimensional, the considered situation can be visualized with the help of a two-dimensional periodic image. Figure \ref{fig:dft_blobs} shows a covering of the torus by three closed toric bodies, centered at the row vectors of the standard DFT matrix, each in a different color. The picture also shows the rectangular grid $\mathcal{G}$ and the periodic structure is evident.

\begin{figure}[h] 
    \caption{Visualization of the grid $\mathcal{G}$ and three closed toric bodies centered at the rows of the DFT matrix, that cover $\mathbb{T}$. Black rectangle represents the fundamental domain $[-\pi, \pi)^2$.} \label{fig:dft_blobs}
    \centering
    \includegraphics[width=.9\linewidth]{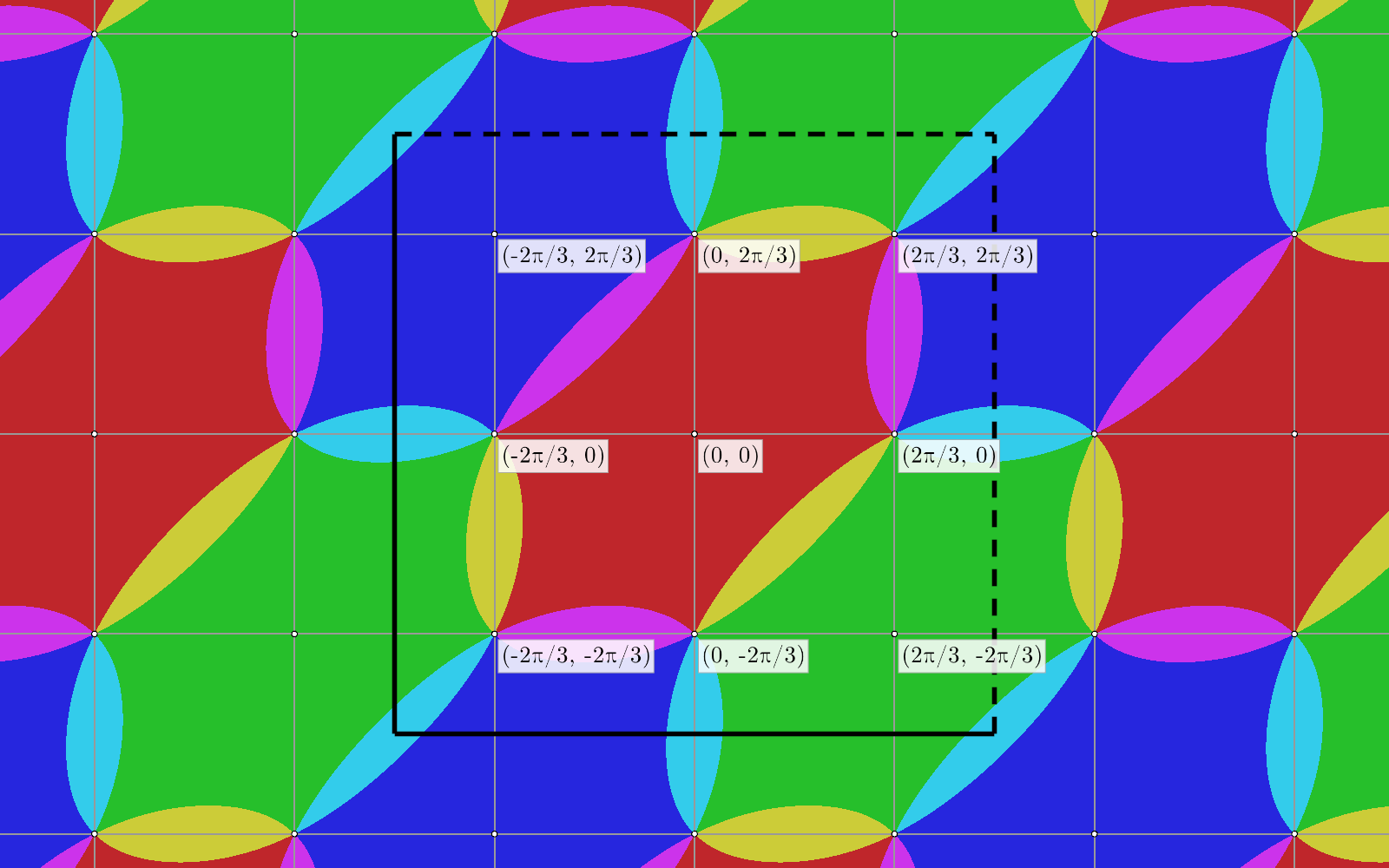}
\end{figure}


Of the $9$ points of the grid $\mathcal{G}$, $6$ lie on the boundary of each toric body. Thus, if the toric bodies are assumed to be open, some grid points remain uncovered. One subtle, non-obvious feature of this visualization is a hidden symmetry. Namely, the situation exhibits complete symmetry with respect to the vertical and horizontal directions, as well as the diagonal direction $x=y$, i.e., the same properties hold for each of these three directions. For instance, consider the problem of covering segments formed by three points from $\mathcal{G}$ with a single toric body. Segments parallel to these three directions are equally easy to cover, this can even be done with some slack. For illustration, observe Figure \ref{fig:b005} where the segment along the $x=y$ diagonal is easily covered, and Figure \ref{fig:vert_strip_cover} shows segments along the vertical direction covered this way. The same holds for the horizontal direction. In contrast, segments of three grid points parallel to the other diagonal, $x=-y$, are impossible to cover by a single toric body. This difference can be explained by the fact that, when converted back to the language of vectors in $\mathbb{C}^3$, the three grid points determining a segment parallel to the diagonal $x=-y$ correspond to an orthogonal basis for $\mathbb{C}^3$.


Now, we shall examine the possible shapes of toric bodies in greater detail. It is clear that multiplying the coordinates of the center $a$ by complex numbers of modulus $1$ merely translates the toric body on the torus. Therefore, all unimodular toric bodies look exactly the same and differ only by a translation. However, non-unimodular toric bodies can take significantly different forms. Let us consider a vector $a = (t, x,y) \in \mathbb{C}^3$ with $|x|=|y|=1$ and $|t| \in [0,1]$. When $|t|=1$, we get a unimodular body, while for $|t|=0$ the toric body $B(a)$ becomes a rectangular band wrapping around the torus. The intermediate values of $|t|$ give a continuous deformation of the unimodular body towards a strip. For sufficiently small $|t|$, the toric body $B(a)$ is non-contractible (it wraps around the torus). Examples are shown in Figures \ref{fig:blobscompa} and \ref{fig:blobscompb}.


\begin{figure}[h!] 
    \centering 
    
    \begin{subfigure}[b]{0.45\linewidth}
        \centering
        \includegraphics[width=\linewidth]{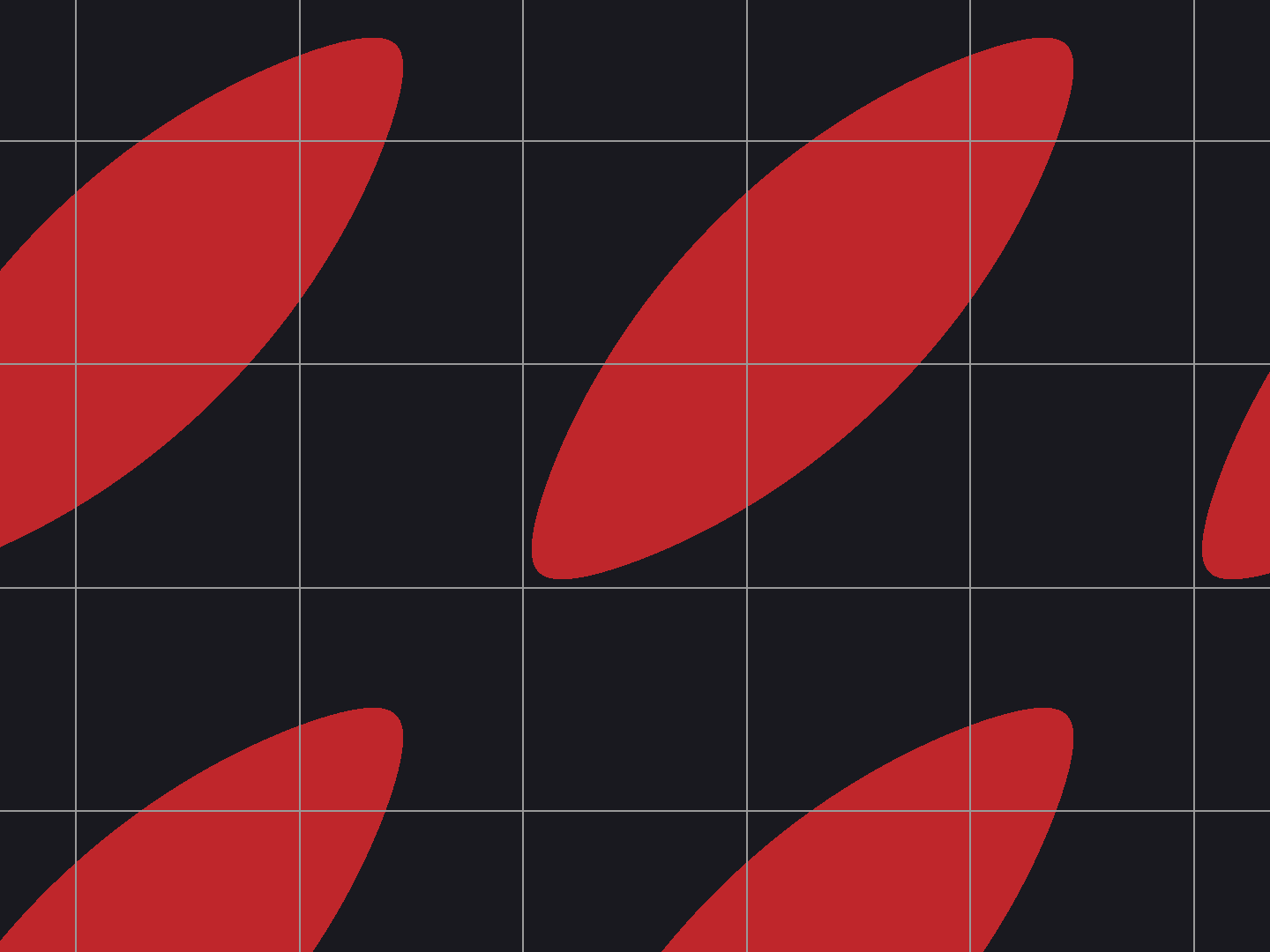}
        \caption{$|t|=0.36$}\label{fig:blobscompa}
        \label{fig:b036}
    \end{subfigure}
    \hfill 
    \begin{subfigure}[b]{0.45\linewidth}
        \centering
        \includegraphics[width=\linewidth]{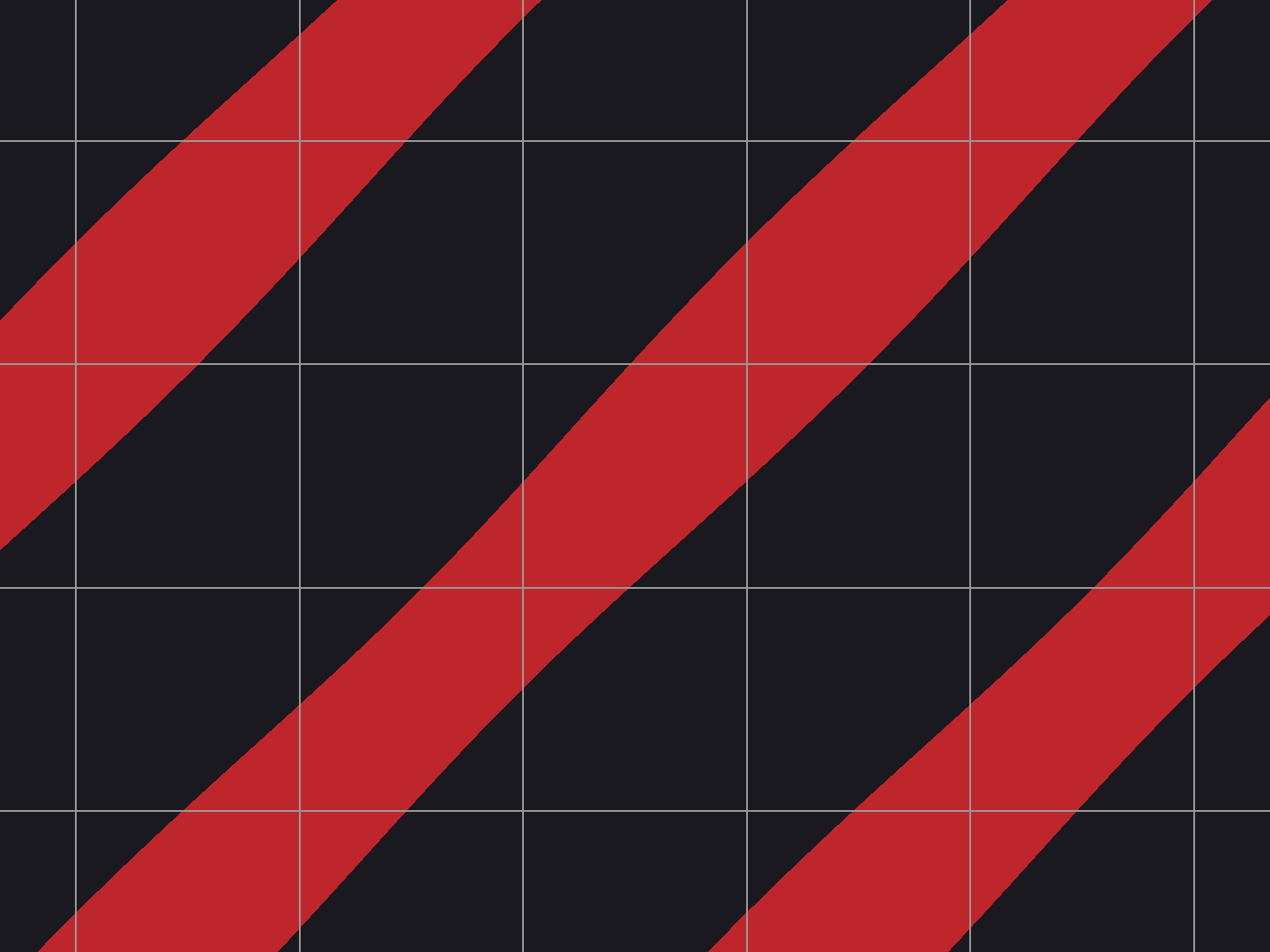}
        \caption{$|t|=0.05$}\label{fig:blobscompb}
        \label{fig:b005}
    \end{subfigure}
    
    \caption{Comparison of toric bodies for different values of $|t| $.}
    \label{fig:blobs_comparison}
\end{figure}

It is not hard to guess that strip-like closed toric bodies can also cover the torus. An example using three vertical strips is shown in Figure \ref{fig:long_blobs}. Analogous coverings exist using horizontal strips or strips parallel to the diagonal $x=y$.

\begin{figure}[h] 
    \caption{Covering of the torus by three strip-like closed toric bodies $\overline{B(a_j)}$ where $a_j = (1, \omega^j, 0)$ for $j=1,2,3$.}\label{fig:long_blobs}
    \centering
    \label{fig:vert_strip_cover}
    \includegraphics[width=.9\linewidth]{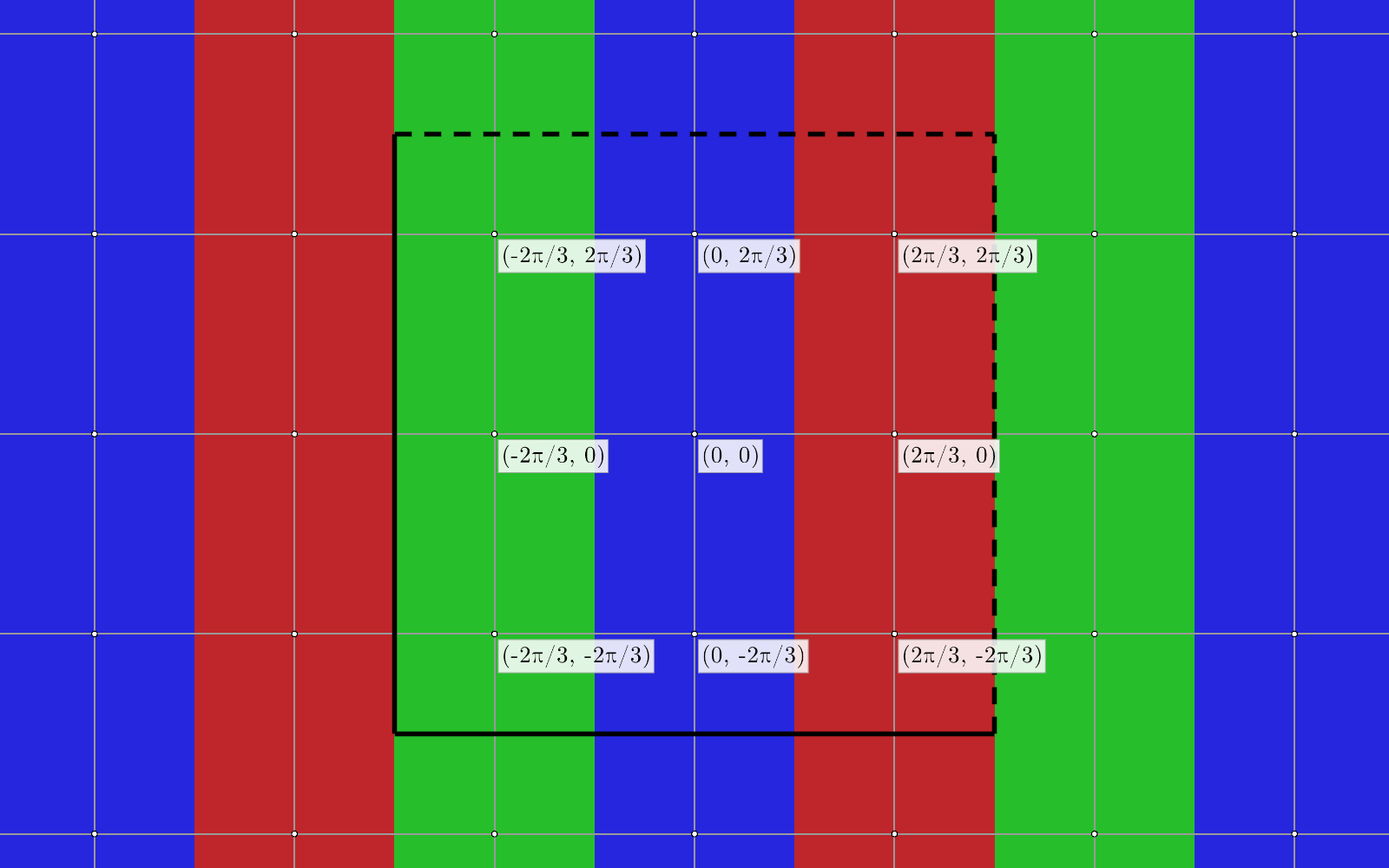}
\end{figure}

This corresponds to an equality case in Conjecture \ref{conj1} for $n=3$ with non-unimodular vectors $a_1, a_2, a_3 \in \mathbb{C}^3$. Specifically, let us take $a_j=(1, \omega^j, 0)$ for $j = 1, 2, 3$, where $\omega = e^{\frac{2\pi}{3}}= -\frac{1}{2} + \frac{\sqrt{3}}{2} i$ denotes the third root of unity. For a unimodular vector $x = (x_1, x_2, x_3)$, the condition $|\langle  x, a_j \rangle | \geq \sqrt{3}$ is equivalent to
$$|x_1|^2 + |x_2|^2 + 2 \re(x_1 \overline{x_2} \omega^j) \geq 3,$$
or just $\re(x_1 \overline{x_2} \omega^j) \geq \frac{1}{2}$. The set of complex numbers $z$ satisfying $|z|=1$ and $\re(z) \geq \frac{1}{2}$ is an arc of the complex unit circle of the angular measure $\frac{2}{3}\pi$. Clearly, any point on the unit circle can be rotated by $0^{\circ}$, $120^{\circ}$ or $240^{\circ}$ (which corresponds to multiplying by $1$, $\omega$ or $\omega^2$, respectively), so that it will belong to this arc. Hence, the inequality $|\langle  x, a_j \rangle | \geq \sqrt{3}$ will be satisfied for some $j \in \{1, 2, 3\}$. 

One may now ask what other shapes of a toric body are possible, when its center $a=(t, x, y)$ has two or more coordinates with modulus less than $1$. It turns out, that such a situation can be omitted from consideration, as the toric body just gets smaller in this case. This is proved in the following lemma, which shows that every toric body is contained in a toric body with the center having at most one coordinate of modulus less than $1$.

\begin{lemma}\label{lem:blob_expansion}
Let $a \in \mathbb{C}^3$ be a non-zero vector such that $\|a\|_{\infty} \leq 1$. Then, there exists a vector $v \in \mathbb{C}^3$ with at most one coordinate of modulus less than $1$ and such that $B(a) \subseteq B(v)$.
\end{lemma}
\begin{proof}

Since $B(a) \subseteq B(ta)$ for any $t>1$, we may assume that $\|a\|_{\infty}=1$. Let us hence suppose that $a = (a_1, a_2, a_3)$, where $|a_1| \leq |a_2| \leq |a_3|=1$. If $a_2=0$, then also $a_1=0$ and thus $B(a) = \emptyset$, so there is nothing to prove. Therefore, let us suppose that $a_2 \neq 0$. We shall prove that the desired vector $v$ for which $B(a) \subseteq B(v)$ can be chosen as $v = \left (a_1, \frac{a_2}{|a_2|}, a_3 \right)$. This is equivalent to showing that for any unimodular $ x = (x_1, x_2, x_3) \in \mathbb{C}^3$ we have
$$|\langle a, x \rangle | > \sqrt{3} \implies |\langle v, x \rangle | > \sqrt{3}.$$
Let us write $x=(x_1, x_2, x_3)$ with $|x_i|=1$ for $i=1, 2 ,3$, $|a_2|=b \leq 1$ and $\frac{a_2}{b}=u$. For $b=1$ there is nothing to show, so we can assume that $b<1$. Let us now define a function $f: [b, 1] \to \mathbb{R}$ as
$$f(t) = \left |\left \langle (a_1, tu, a_3), (x_1, x_2, x_3) \right \rangle \right |^2$$
and let us suppose that $f(b)>3$. We calculate
$$f(t) = \left | a_1 \overline{x_1} + tu \overline{x_2} + a_3 \overline{x_3} \right |^2 = |a_1|^2 + |t|^2 + 1 +2\re \left ( t \overline{a_1}x_1u\overline{x_2} + t\overline{a_3}x_3u\overline{x_2} + a_1\overline{x_1}\overline{a_3}x_3 \right )$$
$$= t^2  + 2At + \left ( |a_1|^2 + 1 + 2\re(a_1\overline{x_1}\overline{a_3}x_3 ) \right ),$$
where $A =\re \left ( \overline{a_1}x_1u\overline{x_2} + \overline{a_3}x_3u\overline{x_2} \right )$. To prove that $f(1)>3$, it would be enough to show that the quadratic function $f$ is increasing on $[b, 1]$. Since $f'(t)=2(t+A)$, this is equivalent to $b + A \geq 0$. Suppose for contradiction that $b < - A$. Then, since $|a_1| \leq b$ and
$$|\re(a_1\overline{x_1}\overline{a_3}x_3 )| \leq |a_1\overline{x_1}\overline{a_3}x_3| = |a_1| \leq b<1$$
we have
$$f(b) = b^2 + 2Ab + \left ( |a_1|^2 + 1 + 2\re(a_1\overline{x_1}\overline{a_3}x_3 ) \right ) < b^2 - 2b^2 + b^2 + 1 + 2 = 3, $$
which contradicts the assumption. Thus, $f$ is non-decreasing on $[b, 1]$. It follows that $f(1)>f(b)>3$, as desired.
\end{proof}

In the next lemma we begin to relate toric bodies to the grid $\mathcal{G}$. To prove the impossibility of a covering of the torus, we have a certain important degree of freedom, as the grid can be arbitrarily translated on the torus, or equivalently, all three toric bodies can be simultaneously translated, while keeping $\mathcal{G}$ fixed. As we shall show, after a suitable translation any given toric body can be made to cover at most one point of $\mathcal{G}$.

\begin{lemma} \label{lem:grid_placement}
For any nonzero vector $a=(a_1, a_2, a_3) \in \mathbb{C}^3$ with $\|a\|_{\infty} \leq 1$, there exists a unimodular vector $x=(x_1, x_2, x_3) \in \mathbb{C}^3$, such that for $v=(a_1x_1, a_2x_2, a_3x_3) \in \mathbb{C}^3$ the toric body $B(v)$ contains at most one point of the grid $\mathcal{G}$.
\end{lemma}

\begin{proof}

By Lemma \ref{lem:blob_expansion} we can assume that at most one coordinate of $a$ has modulus strictly less than $1$. Without loss of generality, let $|a_1|=b \leq |a_2|=|a_3|=1$. We shall prove that the desired vector $v$ can be chosen as $v=(b, u, \overline{u})$, where $u=\exp{\left (i\left ( \arccos{\frac{b}{2}} - \frac{\pi}{3} \right ) \right )}$. In this case, we will show that $(0, 0)$ is the only point of $\mathcal{G}$ that could possibly belong to $B(v)$.

For fixed $k, l \in \{0, 1, 2\}$ we have
$$\left | \left \langle (b, u, \overline{u}), (1, \omega^k, \omega^l) \right \rangle \right |^2 = b^2 + 2 + 2 \re\left ( b u \omega^{-k} + b u \omega^{l} + u^2 \omega^{-k} \omega^{l} \right ).$$
Thus, by changing $-k$ to $k$, what we need to prove can be simply restated as
\begin{equation}
\label{desiredineq}
b^2 + 2b \re \left (u \left ( \omega^{k} + \omega^{l} \right ) \right ) + 2 \re \left (u^2\omega^k\omega^l \right ) \leq 1
\end{equation}
for $k, l \in \{0, 1, 2\}$ and $(k, l) \neq (0, 0)$. To this end, we shall consider some different cases.

\begin{enumerate}
\item $(k, l)=(1, 2)$ or $(k, l) = (2, 1)$. Then $\omega^k + \omega^l=-1$ and $\omega^k \omega^l = 1$, so the desired inequality \eqref{desiredineq} takes the form
$$b^2 - 2b\re(u) + 2\re(u^2) \leq 1.$$
By using the definition of $u$, the fact that $b \in [0, 1]$ and some standard trigonometric formulas, we obtain that
$$b^2 - 2b\re(u) + 2\re(u^2) = b^2 - 2b \cos \left ( \arccos{\frac{b}{2}} - \frac{\pi}{3} \right ) + 2 \cos \left (2\arccos{\frac{b}{2}} - \frac{2\pi}{3} \right ) $$
$$=b^2 - 2b \left( \frac{b}{4} + \frac{\sqrt{3}}{2} \sqrt{1 - \frac{b^2}{4}} \right) + 2 \left( -\frac{1}{2} \left( \frac{b^2}{2} - 1 \right) + \frac{\sqrt{3}}{2} \sqrt{b^2 - \frac{b^4}{4}} \right)$$
$$=b^2 - \frac{b^2}{2} - b\sqrt{3\left ( 1 - \frac{b^2}{4} \right )} + \left ( 1 -\frac{b^2}{2} \right ) + \sqrt{3\left ( b^2 - \frac{b^4}{4} \right )} = 1.$$
So, in this case, the expression turns out to be actually equal to $1$ for all $b \in [0, 1]$.
\item $(k, l) = (1, 0)$ or $(k, l) = (0, 1)$. Then $\omega ^k + \omega^l = 1 + \omega = -\omega^2$ and $\omega^k \omega^l = \omega$, so the inequality \eqref{desiredineq} takes the form
$$b^2  - 2b \re \left (u \omega^2 \right ) + 2 \re \left (u^2\omega \right ) \leq 1.$$
Reasoning similarly as in the previous case, we get
$$b^2  - 2b \cos \left (\arccos{\frac{b}{2}} + \pi \right ) + 2 \cos \left (2\arccos{\frac{b}{2}} \right ) = b^2  + 2b \left ( \frac{b}{2} \right ) + 2 \left ( \frac{b^2}{2} - 1 \right ) = 3b^2 - 2,$$
which is at most $1$, since $b \in [0, 1]$.
\item $(k, l) = (2, 0)$ or $(k, l) = (0, 2)$. Then $\omega ^k + \omega^l = 1 + \omega^2 = -\omega$ and $\omega^k \omega^l = \omega^2$, so the inequality \eqref{desiredineq} can be written as
$$b^2  - 2b \re \left (u \omega \right ) + 2 \re \left (u^2\omega^2 \right ) \leq 1.$$
Reasoning similarly as before, we get
$$b^2  - 2b \cos \left (\arccos{\frac{b}{2}} + \frac{\pi}{3} \right ) + 2 \cos \left (2\arccos{\frac{b}{2}}+\frac{2\pi}{3} \right )$$
$$=b^2 - 2b \left( \frac{b}{4} - \frac{\sqrt{3}}{2} \sqrt{1 - \frac{b^2}{4}} \right) + 2 \left( -\frac{1}{2} \left( \frac{b^2}{2} - 1 \right) - \frac{\sqrt{3}}{2} \sqrt{b^2 - \frac{b^4}{4}} \right)$$
$$=b^2 - \frac{b^2}{2} + b\sqrt{3\left ( 1 - \frac{b^2}{4} \right )} - \frac{b^2}{2} + 1 - b\sqrt{3\left ( 1 - \frac{b^2}{4} \right )} = 1,$$
so again the expression turns out to be constantly equal to $1$. 

\item $(k, l)=(1, 1)$. Then \eqref{desiredineq} rewrites as
$$b^2 + 4b\re(u\omega) + 2 \re(u^2 \omega^2) \leq 1.$$
Again, we calculate
$$b^2 + 4b\re(u\omega) + 2 \re(u^2 \omega^2) = b^2 + 4b \cos \left ( \arccos{\frac{b}{2}} + \frac{\pi}{3}  \right ) + 2 \cos \left ( 2\arccos{\frac{b}{2}} + \frac{2\pi}{3}  \right )$$
$$=b^2 + 4b \left ( \frac{b}{4} - \frac{\sqrt{3}}{2} \sqrt{1 - \frac{b^2}{4}} \right ) + 2 \left ( - \frac{1}{2}  \left ( \frac{b^2}{2} - 1 \right )  -  \frac{\sqrt{3}}{2} \sqrt{b^2 - \frac{b^4}{4}} \right )$$
$$= \frac{3}{2}b^2 - 3b\sqrt{3\left ( 1 - \frac{b^2}{4} \right )} + 1 = \frac{3}{2} b \left ( b - \sqrt{12-3b^2} \right ) + 1.$$
This is clearly at most $1$, since $b \leq 1 < 3 \leq \sqrt{12-3b^2}$.
\item $(k, l)=(2, 2)$. In this case, \eqref{desiredineq} is equivalent to
$$b^2 + 4b\re(u\omega^2) + 2 \re(u^2 \omega) \leq 1.$$
Repeating a similar calculation as in the previous cases leads to
$$b^2 + 4b\re(u\omega^2) + 2 \re(u^2 \omega) = b^2 + 4b\cos \left ( \arccos{\frac{b}{2}} + \pi  \right ) + 2 \cos \left ( 2\arccos{\frac{b}{2}} \right )$$
$$=b^2 + 4b \left ( -\frac{b}{2} \right ) + 2 \left ( \frac{b^2}{2}-1 \right )=-2,$$
so in this situation we end up with a function constantly equal to $-2$.

\end{enumerate}

We have analyzed all possible cases for the pairs $(k, l)$, and therefore, the proof is finished.
\end{proof}

We note that, with the notation from the previous lemma, the toric body $B(v)$ contains the grid point determined by $(k, l) = (1, 1)$ if and only if $b>0$. Hence, when $b=0$, this open toric body contains none of the grid points. The case of $b=0$ corresponds to a strip-like toric body with the grid points occurring only on the boundary of the strip (see Figure \ref{fig:long_blobs}).

Lemma \ref{lem:blob_expansion} allows us to assume that the toric bodies we consider have at most one coordinate of modulus less than $1$. To help with case analysis, we introduce the following classification.

\begin{definition}
    Let $a = (a_1, a_2, a_3) \in \mathbb{C}^3$. We say that the toric body $B(a)$ is of \emph{type $j$} for $j \in \{1, 2, 3\}$ if $|a_j| = \min \{|a_1|, |a_2|, |a_3 \}$.
\end{definition}

Note that unimodular toric bodies are of all three types simultaneously, while the others we consider will have a unique type (since we will always assume that at most one coordinate is of modulus less than $1$).

We have already noted that strip-like toric bodies can cover three-point segments oriented along each of the three cardinal directions: horizontal, vertical, and $x=y$. However, covering three points in each of these lines forces the body to be of a specific type. Specifically, a body covering the diagonal direction must be of type $1$, a body covering the horizontal direction must be of type $2$, and a body covering the vertical direction must be of type $3$. Due to the symmetries involved, it suffices to establish this fact for the horizontal direction alone, which we do in the following lemma.

\begin{lemma}
\label{lem:line}
Let $a=(a_1, a_2, a_3) \in \mathbb{C}^3$ be a vector such that $\|a\|_{\infty} \leq 1$ and
$$\left \{ \left (1, 1, 1 \right), \left (1, \omega, 1 \right), \left (1, \omega^2, 1 \right) \right \} \subseteq B(a).$$
Then $|a_2|<1$.
\end{lemma}
\begin{proof}

Suppose, for the sake of contradiction, that $|a_2| = 1$. By assumption, for every $k \in \{0, 1, 2\}$ we have
$$3 < \left | \left\langle a, \left (1, \omega^k, 1 \right ) \right \rangle  \right |^2=\left | a_1 + a_2 \omega^{-k} + a_3  \right |^2 = \left | z + a_2 \omega^{-k} \right |^2 = |z|^2 + 2\re(z \overline{a_2}\omega^k)+ 1,$$
where $z = a_1+a_3$. Since every point on the complex unit circle can be rotated by either $0^{\circ}$, $120^{\circ}$ or $240^{\circ}$ such that its real part will be at most $-\frac{1}{2}$, there exists some $k \in \{0, 1, 2\}$, for which we have 
$$\re(z \overline{a_2}\omega^k) \leq - \frac{1}{2} |z|\cdot |a_2| = - \frac{|z|}{2}.$$
For this $k$ it follows that
$$3 < |z|^2 + 2\re(z \overline{a_2}\omega^k)+ 1 < |z|^2 - |z| + 1.$$
However, by the triangle inequality we have $|z|=|a_1+a_3| \leq 2$ and for any $t \in [0, 2]$ we have $t^2 - t \leq 2$. Thus, the assumption $|a_2| = 1$ leads to a contradiction, so $|a_2| < 1$.
\end{proof}

Before proving Conjecture \ref{conj1} for $n=3$ we shall also need the following technical, trigonometric lemma. 

\begin{lemma}
\label{lem:trig2}
 Let $b \in [0, 1]$ and $r \in [-1, 1] \setminus \left ( -\frac{1}{2}, \frac{1}{2} \right ).$ Then there does not exist an angle $\gamma \in [0, 2\pi)$ such that $x = \cos \gamma$ and $y= \cos(\gamma - 120^{\circ})$ are of opposite signs and satisfy both of the inequalities
$$4x^2 + 4brx + b^2 > 3 \quad \text{ and } \quad 4y^2 + 4bry + b^2 > 3.$$
\end{lemma}
\begin{proof}
Suppose, for the sake of contradiction, that such an angle $\gamma$ exists. Note that if $(x, y)$ works for $r$, then $(-x, -y)$ works for $-r$ (since $-x = \cos(\gamma + 180^\circ)$ and $-y = \cos(\gamma - 120^\circ + 180^\circ)$). Thus, without loss of generality, we may assume $r \geq \frac{1}{2}$. Moreover, we may also suppose that $x \geq 0 \geq y$.

Solving the quadratic inequalities for $x$ and $y$ gives $ x > A$ and  $y <  B$, where
$$A=\frac{-br+\sqrt{(r^2-1)b^2+3}}{2} \quad \text{ and } \quad  B = \frac{-br-\sqrt{(r^2-1)b^2+3}}{2}. $$
Since $y \geq -1$, we must also have $B \geq -1$, which is equivalent to
\begin{equation}
\label{ineqr}
br \leq \frac{b^2+1}{4}.
\end{equation}
On the other hand, the condition $A \geq \frac{1}{2}$ is easily checked to be equivalent to
$$br \leq \frac{2-b^2}{2}.$$
However, this follows from \eqref{ineqr}, as for $b \in [0, 1]$ we have $\frac{b^2+1}{4} \leq \frac{2-b^2}{2}$. Hence, we conclude that $A \geq \frac{1}{2}$.

Since $x = \cos \gamma$ and $y = \cos (\gamma - 120^{\circ})$ the inequalities $x > A$, $y < B$ can be equivalently restated as
$$-\arccos A < \gamma < \arccos A$$
and
$$\gamma - \frac{2\pi}{3} \not \in \left [ -\arccos B, \arccos B \right ].$$
Since $A > B$ and $\arccos$ is decreasing, we have $\arccos A < \arccos B$. To derive a contradiction with the above conditions on $\gamma$, it suffices to show 
$$- \arccos A - \frac{2 \pi}{3} \geq -\arccos B,$$
or equivalently,
$$\arccos B \geq \arccos A + \frac{2\pi}{3}.$$
We note that both sides of this inequality lie in $[0, \pi]$: the left by the definition of $\arccos$ and the right because $A \geq \frac{1}{2}$ implies $\arccos A \leq \frac{\pi}{3}$, so $\arccos A + \frac{2\pi}{3} \leq \pi$. Since $\cos$ is decreasing on $[0, \pi]$, the desired inequality is therefore equivalent to 
$$B = \cos (\arccos B) \leq \cos \left ( \arccos A + \frac{2\pi}{3} \right ) = -\frac{A}{2} - \frac{\sqrt{3}}{2} \sqrt{1-A^2},$$
which can be rewritten as
$$\sqrt{3(1-A^2)} \leq -A-2B.$$
Clearly, both sides are non-negative since $A+B = -br \leq 0$ and $B \leq 0$. Squaring yields
$$3-3A^2 \leq A^2 + 4AB + 4B^2,$$
or just
$$3 \leq 4(A^2+AB + B^2).$$
We calculate
$$4(A^2+AB + B^2) =  2 \left ( b^2r^2 + (r^2-1)b^2 + 3 \right ) + \left (b^2r^2 - ((r^2-1)b^2 + 3) \right )$$
$$=4b^2r^2 - b^2 + 3 = b^2(4r^2-1) + 3,$$
which is indeed at least $3$, as we have $4r^2-1 \geq 0$ by the assumption. This concludes the proof.

\end{proof}

We are ready to move to the main part of the proof.

\begin{proof}[Proof of Conjecture \ref{conj1} for $n=3$]

Let us assume that there exist vectors $a_1, a_2, a_3 \in \mathbb{C}^3$ satisfying $\|a_i\|_{\infty} \leq 1$ for $i=1, 2, 3$ and such that for any unimodular vector $x \in \mathbb{C}^3$ the inequality $| \langle a_i, x \rangle | > \sqrt{3}$ holds for at least one $i \in \{1, 2, 3\}$. As noted earlier, this is equivalent to $\mathbb{T} \subseteq B(a_1) \cup B(a_2) \cup B(a_3)$. We note that this condition is rotation-invariant, i.e., everything is unchanged, when the $j$-th coordinate of each vector $a_i$ is multiplied by the same number (where $j \in \{1, 2, 3\}$). Thus, by Lemma \ref{lem:grid_placement}, we may assume that the toric body $B(a_3)$ covers at most one point of $\mathcal{G}$. Our goal is therefore to prove that the union $B(a_1) \cup B(a_2)$ cannot cover all of the remaining $8$ points of $\mathcal{G}$. To this end, we shall fully characterize all $4$-element subsets of $\mathcal{G}$ that are coverable by a single toric body.

Let $S \subseteq \mathcal{G}$ be a subset of grid points. We consider the following transformations:
\begin{enumerate}
    \item Replace every $(\theta_1, \theta_2) \in S$ with $\left (\theta_1 + \frac{2\pi}{3}, \theta_2 \right )$.
    \item Replace every $(\theta_1, \theta_2) \in S$ with $\left ( \theta_1, \theta_2+\frac{2\pi}{3} \right )$.
    \item Replace every $(\theta_1, \theta_2) \in S$ with $(\theta_2, \theta_1)$.
    \item Replace every $(\theta_1, \theta_2) \in S$ with $(-\theta_1, -\theta_2)$.
    \item Replace every $(\theta_1, \theta_2) \in S$ with $(-\theta_1, \theta_2-\theta_1)$.
    \item Replace every $(\theta_1, \theta_2) \in S$ with $(\theta_1-\theta_2, -\theta_2)$.
\end{enumerate}
We observe that if a subset $S' \subseteq \mathcal{G}$ is obtained from $S \subseteq \mathcal{G}$ by a finite sequence of these transformations, then $S'$ is coverable by a single toric body if and only if $S$ is. Indeed, transformations $1$ and $2$ correspond to multiplying the first or second coordinate of the center of a toric body by $e^{i \frac{2\pi}{3}}$. Transformation $3$ swaps the second and third coordinates and transformation $4$ corresponds to taking the complex conjugate of all coordinates of the center. For transformation $5$ we have
\begin{align*}
\left | \left  \langle (a_1, a_2, a_3), \left (1, e^{i\theta_1}, e^{i\theta_2} \right)  \right \rangle \right | &= \left | \left  \langle (a_1, a_2, a_3), \left (e^{-i \theta_1}, 1, e^{i(\theta_2-\theta_1)} \right)  \right \rangle \right | \\
&=\left | \left  \langle (a_2, a_1, a_3), \left (1, e^{-i \theta_1}, e^{i(\theta_2-\theta_1)} \right)  \right \rangle \right |,
\end{align*}
which again corresponds to swapping the coordinates of the center. A similar argument applies to transformation $6$, and hence, all transformations preserve coverability by a single toric body.

It is clear that the relation of \emph{$S'$ being reachable from $S$} via the six transformations listed above defines an equivalence relation on the set of all $4$-element subsets of $\mathcal{G}$. In Figures~\ref{fig:blob_class1}--\ref{fig:blob_class4} we list all $\binom{9}{4}=126$ four-element subsets of $\mathcal{G}$, grouped into the four equivalence classes induced by this relation.

We claim that only the subsets in Class 1 can be covered by a single toric body. Establishing this suffices to prove the conjecture. Indeed, observe that any subset of $\mathcal{G}$ of cardinality $5$ contains a $4$-element subset belonging to Class 2, Class 3, or Class 4. For instance, consider the first 4-element subset present in Figure \ref{fig:blob_class1} (square in the lower left corner). It can be checked by hand that extending the subset by any single point in the grid (there are 5 choices) will result in the new 5-element subset containing some $4$-element subset of Class 2, 3, or 4. Therefore, this specific Class 1 $4$-element subset cannot be extended to a coverable $5$-element subset. By symmetry, this implies that no single toric body can cover $5$ points of $\mathcal{G}$. 

We now show how this observation can be used to complete the proof. By Lemma \ref{lem:grid_placement}, we may assume that the toric body $B(a_3)$ covers at most one point of $\mathcal{G}$. Therefore, if entire $\mathcal{G}$ is covered by three toric bodies, the union $B(a_1) \cup B(a_2)$ would have to cover at least $8$ points. Since neither body can cover $5$ points, each of them has to cover exactly $4$ points (forcing them to be Class 1), and these two sets of points have to be disjoint. However, it is simple to check by hand that any two subsets from Class 1 share at least one common point. Thus, it is impossible to cover the remaining $8$ points with two toric bodies.

\begin{figure}[h!] 
    \centering
    
    \begin{subfigure}[b]{0.49\linewidth}
        \centering
        \includegraphics[width=\linewidth]{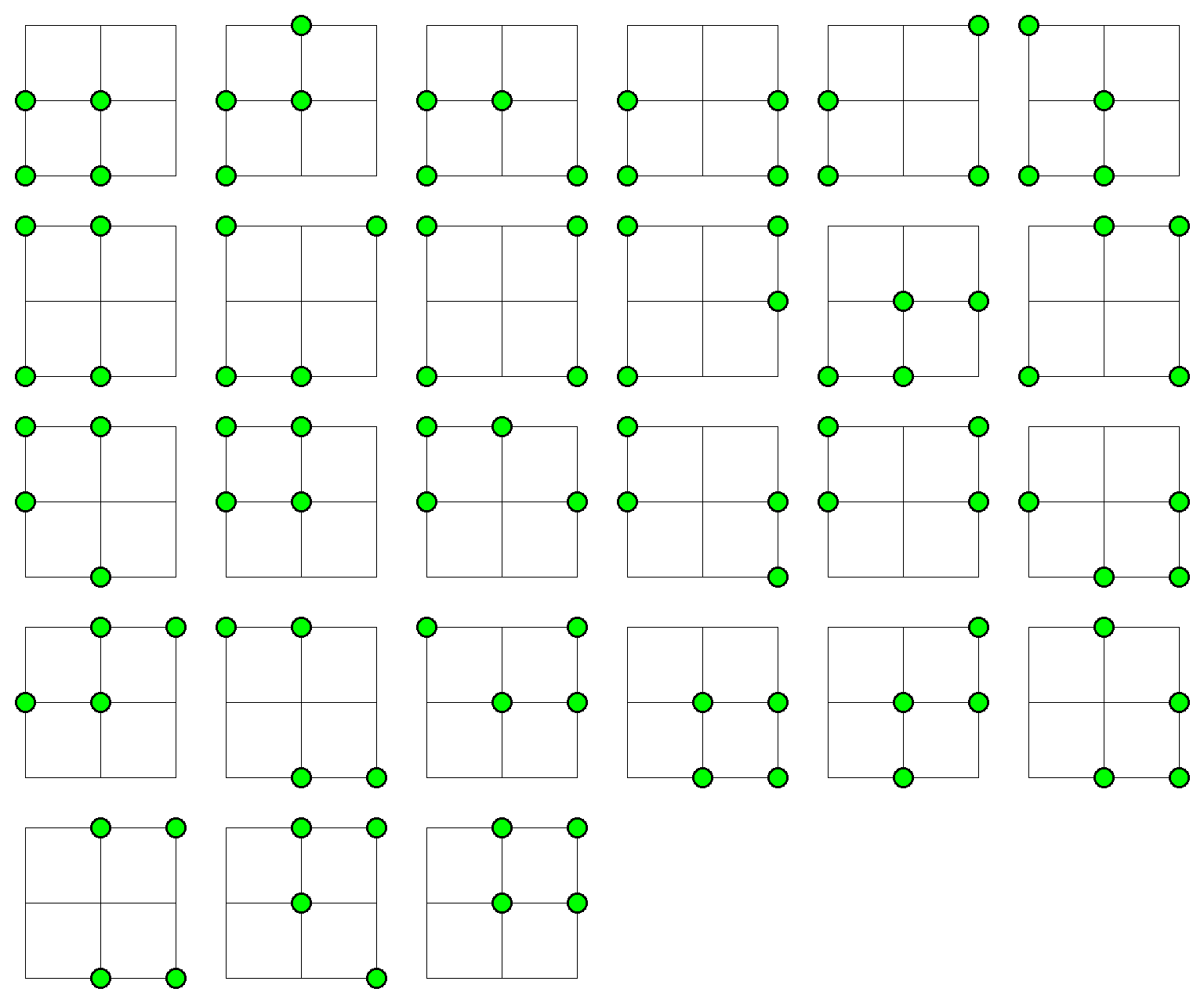}
        \caption{\textbf{Class 1}.}
        \label{fig:blob_class1}
    \end{subfigure}
    \hfill
    \begin{subfigure}[b]{0.49\linewidth}
        \centering
        \includegraphics[width=\linewidth]{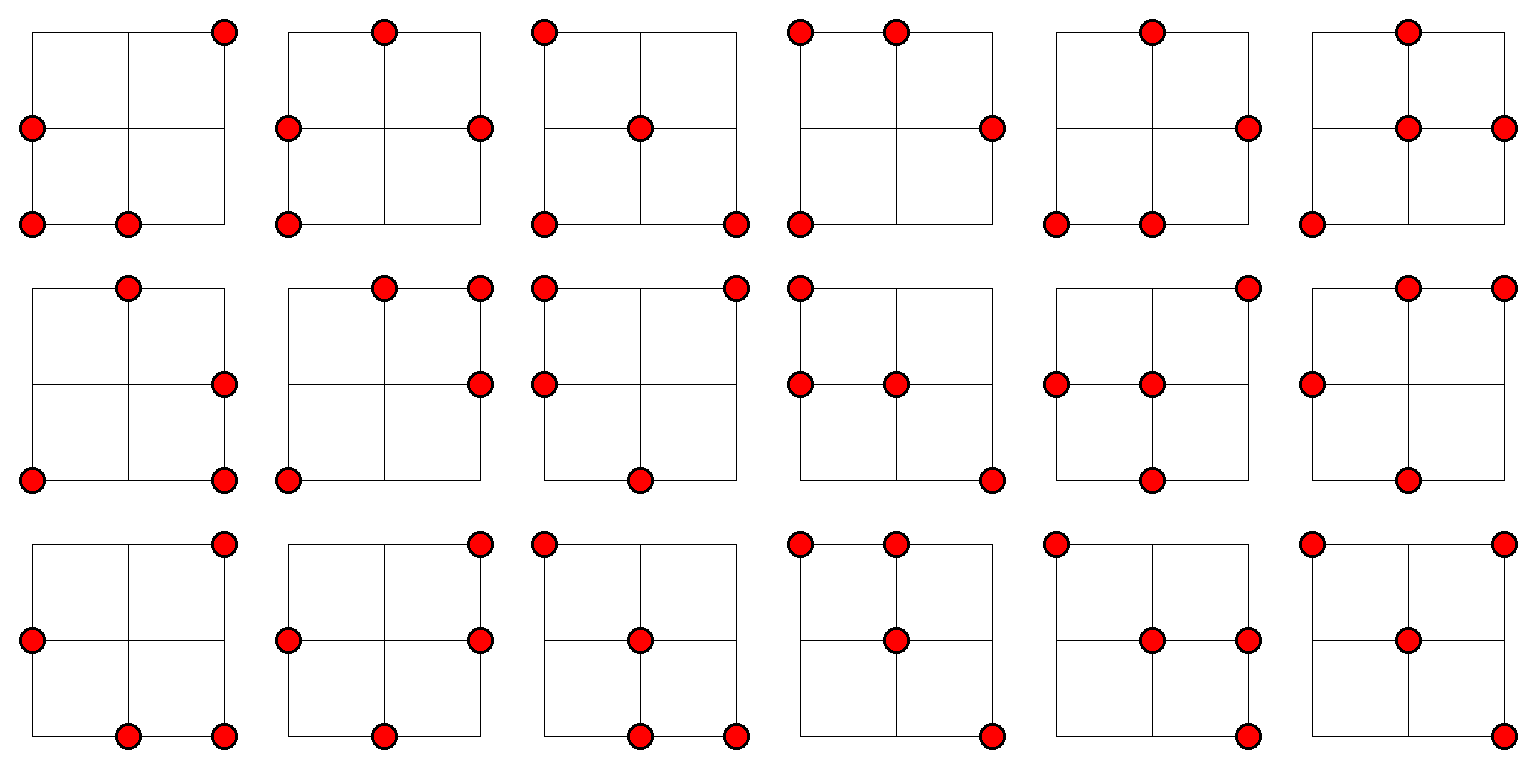}
        \caption{\textbf{Class 2}.}
        \label{fig:blob_class2}
    \end{subfigure}
    
    \vspace{4mm}

    \begin{subfigure}[b]{0.49\linewidth}
        \centering
        \includegraphics[width=\linewidth]{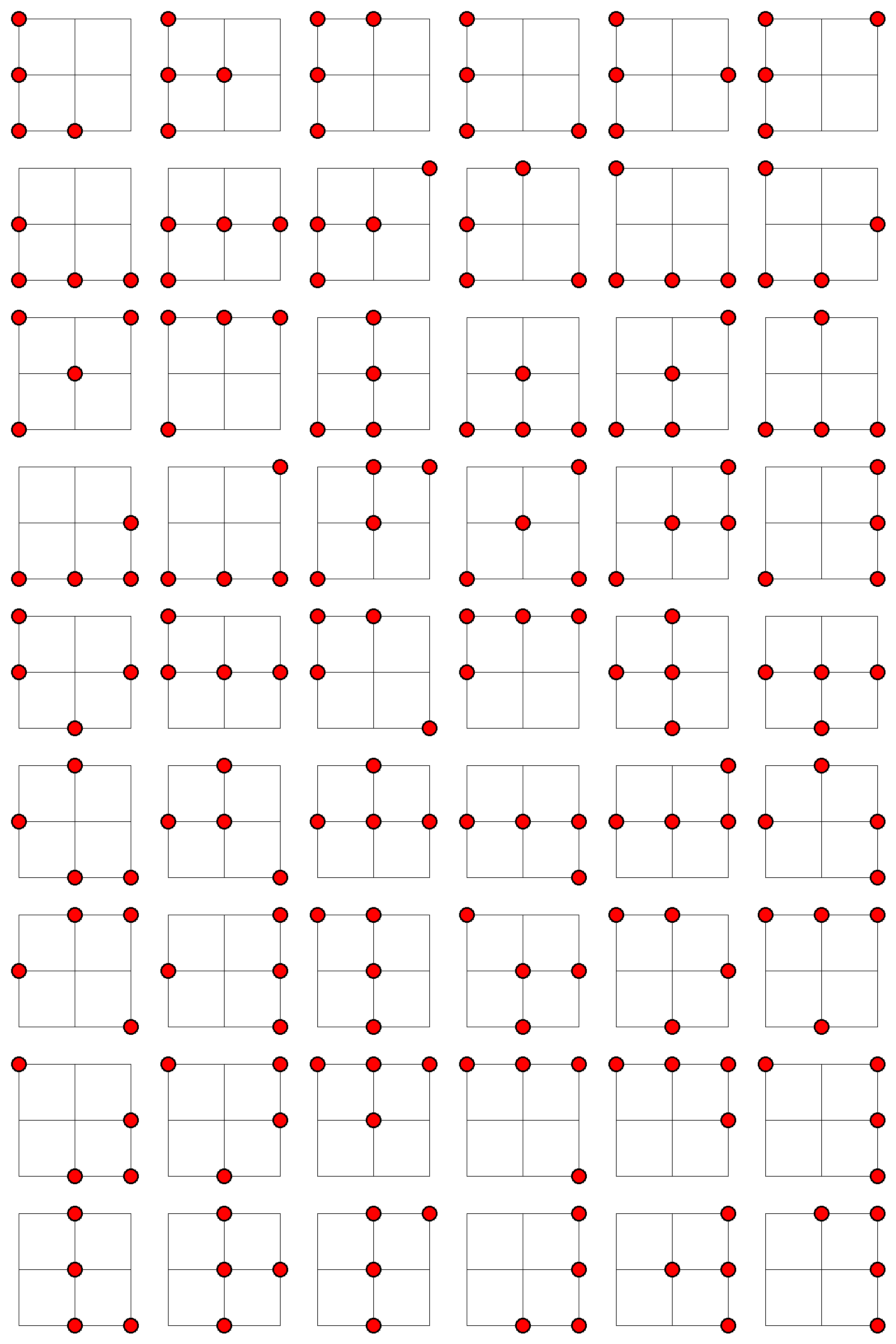}
        \caption{\textbf{Class 3}.}
        \label{fig:blob_class3}
    \end{subfigure}
    \hfill
    \begin{subfigure}[b]{0.49\linewidth}
        \centering
        \includegraphics[width=\linewidth]{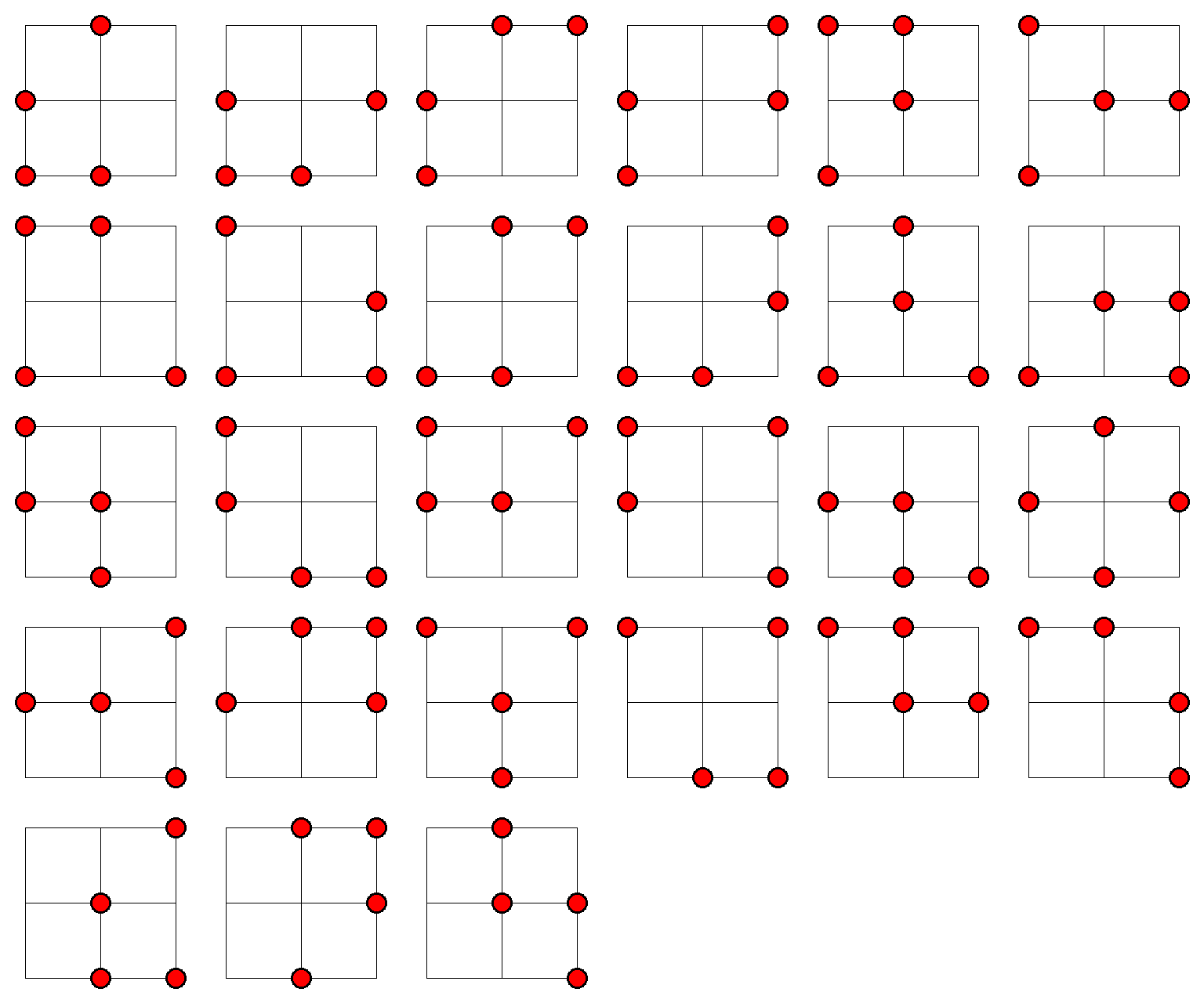}
        \caption{\textbf{Class 4}.}
        \label{fig:blob_class4}
    \end{subfigure}

\end{figure}

The easiest class to rule out is the second one, where every subset contains three points on a line parallel to the diagonal $x=-y$. It turns out that already such a triple cannot be covered by a single toric body. Indeed, such a triple corresponds to an orthogonal basis of $\mathbb{C}^3$ consisting of unimodular vectors $x_1, x_2, x_3$ (for example the row vectors of the $3 \times 3$ DFT matrix). For such $x_1, x_2, x_3 \in \mathbb{C}^3$, the Parseval identity implies that any vector $a \in \mathbb{C}^3$ with $\|a\|_{\infty} \leq 1$ satisfies
$$|\langle a, x_1 \rangle| ^2 + |\langle a, x_2 \rangle| ^2 + |\langle a, x_3 \rangle|^2 = 3 \|a\|_2^2 \leq 9,$$
so that necessarily $|\langle a, x_i \rangle| \leq \sqrt{3}$ for some $i \in \{1, 2, 3\}$, i.e., the toric body $B(a)$ does not contain $x_i$.

The remaining two classes require significantly more effort and it will be necessary to use all four points in a subset. We begin with class $3$, characterized by having three points lying on a single line in one of the cardinal directions (horizontal, vertical, or $x=y$). Since it is enough to analyze only a single representative of Class~$3$, let us take a subset containing three points in a horizontal line, i.e., for example, the points of the form $(1, \omega^k, 1)$ for $k \in \{0, 1, 2\}$ and the point $(1, \omega, \omega)$ (see Figure \ref{fig:subset3}).

\begin{figure}[h] 
    \caption{A considered subset of Class~$3$.}\label{fig:subset3}
    \centering
    \includegraphics[width=.2\linewidth]{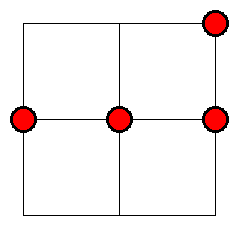}
\end{figure}

Suppose there exists a vector $a \in \mathbb{C}^3$ with $\|a\|_{\infty} \leq 1$ such that the toric body $B(a)$ contains all of those $4$ points. By Lemma \ref{lem:blob_expansion} we can suppose that at most one coordinate of $a$ is of modulus strictly less than $1$. Lemma \ref{lem:line} further implies that this must be the second coordinate, i.e. the toric body $B(a)$ is of type $2$. Hence $a = (e^{i\alpha}, b, e^{i\beta})$ for some $\alpha, \beta \in [0, 2\pi]$ and $b \in [0, 1)$. For $k \in \{0, 1, 2\}$ we have
$$\left | \left \langle (e^{i\alpha}, b, e^{i\beta}), (1, \omega^k, 1) \right \rangle \right |^2 = 2 + b^2 + 2b \re \left ( e^{i\alpha}\omega^k + e^{-i \beta} \omega^{-k} \right ) + 2\re \left ( e^{i(\alpha-\beta)} \right )$$
$$ = 2 + b^2 + 2b \cos \left ( \alpha + \frac{2k\pi}{3} \right ) + 2b \cos \left ( -\beta -  \frac{2k\pi}{3} \right ) +  2 \cos \left ( \alpha - \beta \right ).$$
After changing $\beta$ to $-\beta$ for simplicity, we arrive at the following conditions which need to be satisfied for $k \in \{0, 1, 2\}$:
$$b \left ( \cos \left ( \alpha + \frac{2k\pi}{3} \right ) + \cos \left ( \beta -  \frac{2k\pi}{3} \right ) \right ) + \cos(\alpha + \beta) > \frac{1-b^2}{2}.$$
The condition for the fourth point $(1, \omega, \omega)$ reads similarly as
$$b\left (\cos \left ( \alpha + \frac{2\pi}{3} \right ) + \cos \beta \right ) + \cos \left (\alpha + \beta + \frac{2\pi}{3} \right ) > \frac{1-b^2}{2}.$$
We aim to derive a contradiction from these conditions. To this end, we first rewrite them as
\begin{enumerate}
\item $-2b \cos \gamma \cos \omega  + \cos(2 \gamma) > \frac{1-b^2}{2}$,
\item $2b \cos \gamma \cos (\omega - 60^{\circ}) + \cos(2 \gamma) > \frac{1-b^2}{2}$,
\item $2b \cos \gamma \cos (\omega + 60^{\circ}) + \cos(2 \gamma) > \frac{1-b^2}{2}$,
\item $2b \cos (\gamma-120^{\circ}) \cos (\omega + 60^{\circ}) + \cos(2 \gamma+120^{\circ}) > \frac{1-b^2}{2}$,
\end{enumerate}
where $\gamma = \frac{\alpha+\beta}{2} + \pi$ and $\omega = \frac{\alpha - \beta}{2}$. Then, using $\cos(2 \gamma) = 2\cos^2 \gamma - 1$, we can further rewrite everything as
\begin{enumerate}
\item $4x^2 - 4btx + b^2 > 3$,
\item $4x^2 + 4bzx + b^2 > 3$,
\item $4x^2 + 4bwx + b^2 > 3$,
\item $4y^2 + 4bwy + b^2 > 3$,
\end{enumerate}
where $x = \cos \gamma$, $y = \cos(\gamma - 120^{\circ})$, $t=\cos \omega$, $z=\cos (\omega-60^{\circ})$ and $w = \cos (\omega + 60^{\circ})$. After adding $\pi$ to all of the angles, the variables will change their signs, while the inequalities will stay the same. Hence, without loss of generality, we can suppose that $x \geq 0$. Clearly, for any angle $\omega$ we have
$$\min\{z, -t, w\} = \min \left \{\cos (\omega-60^{\circ}), \cos \left (\omega + 180^{\circ} \right ), \cos (\omega+60^{\circ}) \right \} \leq -\frac{1}{2}.$$
In particular, taking into account the inequalities $1-3$ we see that $x$ has to satisfy the quadratic inequality
\begin{equation}
\label{quadineqx}
4x^2 - 2bx + b^2 - 3 > 0,
\end{equation}
which gives us
\begin{equation}
\label{ineqx}
x > \frac{b+\sqrt{3(4-b^2)}}{4},
\end{equation}
as $x \geq 0$. Now, we shall consider two cases: if $y \geq 0$ or $y<0$.

Let us first assume that $y \geq 0$. Since $w \leq 1$, from the inequality $4$ it follows that $y$ satisfies
$$4y^2 + 4by + b^2 - 3 > 0,$$
which yields $y > \frac{-b+\sqrt{3}}{2}$. From the inequality \eqref{ineqx} if follows easily that $x > \frac{b+3}{4}$, as $b \leq 1$. Moreover, since $x$ and $y$ are cosines of an angles differing by $120^{\circ}$, it is easy to verify that $x^2 + xy + y^2 = \frac{3}{4}$. Combining all these facts together leads to the estimate
$$\frac{3}{4} = x^2 + xy + y^2 \geq \frac{1}{16} \left ((b+3)^2  + 2(b+3)\left ( -b+\sqrt{3} \right )  +  4\left ( -b+\sqrt{3} \right )^2 \right ),$$
that can be rewritten as
$$0 \geq 3 \left (b^2 - 2b\sqrt{3} + 3  + 2\sqrt{3} \right )  = 3 \left ( \left ( b - \sqrt{3} \right )^2 + 2\sqrt{3} \right),$$
but this is obviously not true, so we obtain a contradiction.

Now, let us suppose that $y<0$. Let us first assume also that $w \geq -\frac{1}{2}$. Then $y$ satisfies
$$4y^2 - 2by + b^2 - 3 > 0,$$
which combined with \eqref{quadineqx} gives us a contradiction with Lemma \ref{lem:trig2} for $r:= - \frac{1}{2}$. In the case $w < - \frac{1}{2}$ we get a contradiction with Lemma 
\ref{lem:trig2} for $r:=w$ by considering the quadratic inequalities $3$ and $4$, satisfied by $x$ and $y$ respectively. Thus, the case of Class~$3$ is concluded.

We move to the case of Class~$4$. Subsets in this class are characterized by containing: a segment in each of the three cardinal directions (horizontal, vertical, $x=y$), two segments in direction of the diagonal $x=-y$ and an additional segment in one of the cardinal directions. Since we are free to choose any subset of this class, let us pick one which has two diagonal segments, so that the situation is symmetric for horizontal and vertical directions. An example of such subset is $\{(1, \omega, 1), (1, \omega^2, 1), (1, 1, \omega), (1, 1, \omega^2)\}$ (shown in Figure  \ref{fig:subset4}).

\begin{figure}[h] 
    \caption{A considered subset of Class~$4$.}\label{fig:subset4}
    \centering
    \includegraphics[width=.2\linewidth]{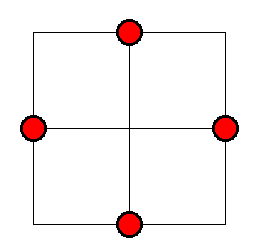}
\end{figure}

Let us first suppose that we are dealing with a toric body of type $2$, i.e. of the form $(e^{i \alpha}, b, e^{i \beta})$. Repeating similar calculations as before, and changing $\beta$ to $-\beta$ in the end, we arrive at the inequalities of the form

$$b \left ( \cos \left ( \alpha + \frac{2k\pi}{3} \right ) + \cos \left ( \beta -  \frac{2k\pi}{3} \right ) \right ) + \cos(\alpha + \beta) > \frac{1-b^2}{2},  $$
and
$$b \left ( \cos \alpha + \cos \left ( \beta +  \frac{2k\pi}{3} \right ) \right ) + \cos \left (\alpha + \beta + \frac{2k\pi}{3} \right ) > \frac{1-b^2}{2},$$
that need to be satisfied for $k \in \{1, 2\}$. We note that if we would consider a toric body of type $3$, i.e. of the form $(e^{i \alpha}, e^{i \beta}, b)$, then we would end up with the exact same set of inequalities, and therefore, it is enough to consider only type $2$. However, toric bodies of type $1$ will require a different consideration, which will be done later. We first rewrite the inequalities as

\begin{enumerate}
\item $2b \cos \gamma \cos \left (\omega + 120^{\circ} \right )+ \cos(2 \gamma) > \frac{1-b^2}{2}$,
\item $2b \cos \gamma \cos \left ( \omega + 240^{\circ} \right ) + \cos(2 \gamma) > \frac{1-b^2}{2}$,
\item $2b \cos (\gamma+60^{\circ}) \cos \left (\omega - 60^{\circ} \right ) + \cos(2 \gamma+120^{\circ}) > \frac{1-b^2}{2}$,
\item $2b \cos (\gamma+120^{\circ}) \cos \left (\omega + 240^{\circ} \right ) + \cos(2 \gamma+240^{\circ}) > \frac{1-b^2}{2}$,
\end{enumerate}
where $\gamma = \frac{\alpha+\beta}{2}$ and $\omega = \frac{\alpha - \beta}{2}$. Then, using $\cos(2 \gamma) = 2\cos^2 \gamma - 1$, we can restate this further as

\begin{enumerate}
\item $4x^2 + 4bwx + b^2 > 3$,
\item $4x^2 + 4btx + b^2 > 3$,
\item $4z^2 - 4bwz + b^2 > 3$,
\item $4y^2 + 4bty + b^2 > 3$,
\end{enumerate}
where $x = \cos \gamma$, $y = \cos(\gamma + 120^{\circ})$, $z=\cos (\gamma + 60^{\circ})$, $t=\cos (\omega+240^{\circ})$, $w = \cos (\omega + 120^{\circ})$. Clearly, after adding $\pi$ to all the angles, the variables will change their signs, but the inequalities will remain unchanged. Therefore, without loss of generality we may assume that $x \geq 0$. We note that $t$ and $w$ can not be both in $\left ( -\frac{1}{2}, \frac{1}{2} \right )$, as there are cosines of angles differing by $120^{\circ}$. Let us first suppose that $t \not \in \left ( -\frac{1}{2}, \frac{1}{2} \right )$. If $y < 0$, then by the inequalities $2$ and $4$ we get an immediate contradiction with Lemma \ref{lem:trig2} for $r:=t$. Thus, let us suppose that $y \geq 0$. If $t \leq - \frac{1}{2}$, then by the inequalities $2$ and $4$ we have
$$x > \frac{-b+\sqrt{3(4-b^2)}}{4} \quad \text{ and } \quad y > \frac{-b+\sqrt{3(4-b^2)}}{4}.$$
However, it is easy to check that for $b \in [0, 1]$ we have
$$\frac{-b+\sqrt{3(4-b^2)}}{4} \geq \frac{1}{2}.$$
and hence $x, y > \frac{1}{2}$, but this is clearly impossible, as $x$ and $y$ are cosines of angles differing by $120^{\circ}$.

Now, suppose that $t > \frac{1}{2}.$ Then $w < \frac{1}{2}$. Moreover, because $x, y \geq 0$ we must have also $z \geq 0$ (as $z=\cos (\gamma + 60^{\circ})$). If $w \leq -\frac{1}{2}$, then we obtain a contradiction with Lemma \ref{lem:trig2} for $r:=w$ combined with the inequalities $1$ and $3$, similarly as in the previous case, but applied to $x$ and $-z=\cos(\gamma - 120^{\circ}) \leq 0$. Thus, let us suppose that $w > -\frac{1}{2}$. Then, analogously as before, we conclude that
$$x > \frac{-b+\sqrt{3(4-b^2)}}{4} \geq \frac{1}{2}$$
and
$$-z > \frac{-b+\sqrt{3(4-b^2)}}{4} \geq \frac{1}{2},$$
which is a contradiction, as $x$ and $-z$ are cosines of angles differing by $120^{\circ}$. This concludes the proof in the considered case.

We are left with considering a possibility of the covering of the same subset of the grid, but for the toric bodies of type $1$, i.e. $(b, e^{i \alpha}, e^{i \beta})$. In this case, assuming that the considered subset is covered, a similar calculation as before yields that the following conditions
$$b \left ( \cos \left ( \alpha + \frac{2k\pi}{3} \right ) + \cos \beta  \right ) + \cos \left( \alpha + \beta+ \frac{2k\pi}{3} \right ) > \frac{1-b^2}{2},  $$
and
$$b \left ( \cos \alpha + \cos \left ( \beta + \frac{2k\pi}{3} \right ) \right )  + \cos \left( \alpha + \beta+ \frac{2k\pi}{3} \right ) > \frac{1-b^2}{2},  $$
need to hold for $k \in \{1, 2\}$. We rewrite them as

\begin{enumerate}
\item $2b \cos \gamma \cos \omega  + \cos(2 \gamma) > \frac{1-b^2}{2}$,
\item $2b \cos \left ( \gamma - 120^{\circ} \right ) \cos (\omega - 120^{\circ}) + \cos(2\gamma - 240^{\circ}) > \frac{1-b^2}{2}$,
\item $2b \cos \gamma \cos \left  (\omega - 120^{\circ} \right ) + \cos(2 \gamma) > \frac{1-b^2}{2}$,
\item $2b \cos (\gamma-120^{\circ}) \cos \omega + \cos(2 \gamma-240^{\circ}) > \frac{1-b^2}{2}$,
\end{enumerate}
where $\gamma = \frac{\alpha+\beta}{2} + 60^{\circ}$ and $\omega = \frac{\alpha - \beta}{2} + 60^{\circ}$. Consequently, by $\cos(2 \gamma) = 2\cos^2 \gamma - 1$ this leads to the system of inequalities
\begin{enumerate}
\item $4x^2 + 4btx + b^2 > 3$,
\item $4x^2 + 4bwx + b^2 > 3$,
\item $4y^2 + 4bty + b^2 > 3,$
\item $4y^2 + 4bwy + b^2 > 3$,
\end{enumerate}
where $x = \cos \gamma$, $y = \cos(\gamma - 120^{\circ})$, $t=\cos \omega$, $w=\cos (\omega-120^{\circ})$. As in the previous cases, we can assume that $x \geq 0$. For any angle $\omega$ we have
$$\min\{t, w\} \leq \frac{1}{2} \quad \text{ and } \quad \max\{t, w\} \geq -\frac{1}{2}.$$
The same holds for $x$ and $y$:
$$\min\{x, y\} \leq \frac{1}{2} \quad \text{ and } \quad \max\{x, y\} \geq -\frac{1}{2}.$$
In particular, taking into account inequalities $1$ and $2$ we see that $x$ has to satisfy the inequality
\begin{equation}
\label{ineqx2}
x > \frac{-b+\sqrt{3(4-b^2)}}{4},
\end{equation}
as $x \geq 0$. Now, we shall consider two cases: $y \geq 0$ or $y<0$.

If $y \geq 0$, then $y$ also satisfies the inequality
$$y > \frac{-b+\sqrt{3(4-b^2)}}{4}.$$
However, as already noted before, for $b \in [0, 1]$ we have
$$\frac{-b+\sqrt{3(4-b^2)}}{4} \geq \frac{1}{2}.$$
It follows that $x, y>\frac{1}{2}$, but this is impossible for the cosines of angles differing by $120^{\circ}$.

Let us now suppose that $y<0$. Since $t, w$ are cosines of angles differing by $120^{\circ}$, they can not be both in the interval $\left [ - \frac{1}{2}, \frac{1}{2} \right ]$. Let us therefore suppose that $w \in [-1, 1] \setminus \left ( -\frac{1}{2}, \frac{1}{2} \right )$. In this case, by inequalities $2$ and $4$, we get a contradiction with Lemma \ref{lem:trig2} for $r:=w$. Similarly, if $t \not \in \left [-\frac{1}{2}, \frac{1}{2} \right ]$, then by inequalities $1$ and $3$ we get a contradiction with Lemma \ref{lem:trig2} for $r:=t$. This concludes the proof.

\end{proof}

We believe that establishing Conjecture \eqref{conj1} for $n \geq 4$ will be highly non-trivial, even in its weaker version, where the vectors $a_1, \ldots, a_n \in \mathbb{C}^n$ are assumed to be unimodular. Under this assumption, the corresponding higher-dimensional toric bodies centered in vectors $a_i$ also differ only by translation. However, for example in dimension $n=4$, there exists a one-parameter infinite family of pairwise non-equivalent complex Hadamard matrices, each achieving equality in \eqref{conj1}. Therefore, since the bound $\sqrt{n}$ in the conjecture is sharp, any proof for $n=4$ must necessarily account for this infinite family. We do not know whether, as in the case $n=3$, there also exist non-unimodular vectors achieving equality. More broadly, we suspect that Conjecture~\ref{conj1} is deeply intertwined with the geometry of $\mathbb{C}^n$, and in particular, with the classification of $n \times n$ complex Hadamard matrices.

\end{document}